\documentclass[11pt, a4paper, twoside]{amsart}
\usepackage{a4, latexsym, amsfonts,epsf, amsmath, amsthm, amssymb, xy, stmaryrd, graphicx, tikz}
\usetikzlibrary{matrix,arrows}
\usepackage[utf8]{inputenc}
\usepackage{cyrillic}
\usepackage{calligra}
\usepackage{fancyhdr}
\xyoption{all}

\newtheoremstyle{swappedplain}{\topsep}{\topsep}%
{\itshape}
{}
{\bfseries}
{.}
{ }
{\thmnumber{(#2)}\thmname{ #1}\thmnote{ #3}}

\newtheoremstyle{swappeddefinition}{\topsep}{\topsep}%
{}
{}
{\bfseries}
{.}
{ }
{\thmnumber{(#2)}\thmname{ #1}\thmnote{ #3}}

\theoremstyle{definition}

\newtheorem{defi}{Definition}[section]
\newtheorem{rema}[defi]{Remark}
\newtheorem{remas}[defi]{Remarks}
\newtheorem{exam}[defi]{Example}

\theoremstyle{plain}

\newtheorem{prop}[defi]{Proposition}

\newtheorem{theo}[defi]{Theorem}
\newtheorem{coro}[defi]{Corollary}
\newtheorem{lemm}[defi]{Lemma}
\newtheorem*{theono}{Theorem}

\renewcommand*{\emptyset}{\varnothing}

\newcommand{\f}[1]{\it #1\rm}							
\newcommand{\llangle}{\langle\!\langle}						
\newcommand{\rrangle}{\rangle\!\rangle}						

\newcommand{\resprod}{\setbox0 = \hbox{$\displaystyle\prod$}			
    \mathop{\vtop{\copy0\kern -1.6pt \hrule}}}					

\newcommand{\prlim}[1]{\lim\limits_{\substack{\longleftarrow\\ #1}}}		
\newcommand{\cd}{\mbox{\it cd}\ }						
\newcommand{\infl}{\mbox{\it inf}\ }						
\newcommand{\tg}{\mbox{\it tg}}							
\newcommand{\tr}{\mbox{\it tr}}							
\newcommand{\Gal}{\operatorname{Gal}}						

\newcommand{\dimfp}{\dim_{\mathbb{F}_p}}
\newcommand{\modulo}{\operatorname{mod}}
\newcommand{\Hom}{\operatorname{Hom}}						
\newcommand{\xyalign}{\entrymodifiers={+!!<0pt,\fontdimen22\textfont2>}}
\newcommand{\textxymatrix}[1]{\xymatrix@C=12pt{#1}}				
\newcommand{\mf}[1]{\mathfrak{#1}}						
\newcommand{\mc}[1]{\mathcal{#1}}						
\newcommand{\gr}{\operatorname{gr}}						
\newcommand{\grg}{\mathbb{F}_p\llbracket G \rrbracket}				
\newcommand{\power}{\mathbb{F}_p\llangle X \rrangle}				
\newcommand{\polyx}{\mathbb{F}_p\langle X\rangle}				
\newcommand{\kpolyx}{k\langle X\rangle}						
\newcommand{\stackrels}[2]{\stackrel{#1}{\mbox{\tiny $#2$}}} 			
\newcommand{\Q}{\mathbb{Q}}							
\newcommand{\N}{\mathbb{N}}							
\newcommand{\Z}{\mathbb{Z}}							
\newcommand{\F}{\mathbb{F}}							
\newcommand{\ol}[1]{\overline{#1}}						

\newcommand{\CC}{\mc{C}}							
\newcommand{\z}{\mf{z}}								

\long\def\symbolfootnote[#1]#2{\begingroup%
\def\thefootnote{\fnsymbol{footnote}}\footnote[#1]{#2}\endgroup}

	\renewcommand{\thefootnote}{}

\begin{document}

\thispagestyle{empty}
\pagestyle{fancy}
\markleft{\it Gärtner, Higher Massey products in the cohomology of mild pro-$p$-groups}
\markright{\it Gärtner, Higher Massey products in the cohomology of mild pro-$p$-groups}

\begin{abstract}
Translating results due to J.\ Labute into group cohomological language, A.\ Schmidt proved that a finitely presented pro-$p$-group $G$ is mild and hence of cohomological dimension $\cd G=2$ if $H^1(G,\F_p)=U\oplus V$ as $\F_p$-vector space and the cup-product $H^1(G,\F_p)\otimes H^1(G,\F_p)\to H^2(G,\F_p)$ maps $U\otimes V$ surjectively onto $H^2(G,\F_p)$ and is identically zero on $V\otimes V$. This has led to important results in the study of $p$-extensions of number fields with restricted ramification, in particular in the case of tame ramification. In this paper, we extend Labute's theory of mild pro-$p$-groups with respect to weighted Zassenhaus filtrations and prove a generalization of the above result for higher Massey products which allows to construct mild pro-$p$-groups with defining relations of arbitrary degree. We apply these results for one-relator pro-$p$-groups and obtain some new evidence of an open question due to Serre.
\end{abstract}

\title{Higher Massey products in the cohomology of mild pro-$p$-groups}
\author{by Jochen Gärtner at Heidelberg}
\date{\today}
\maketitle

\footnote{
\it key words and phrases. \rm Mild pro-$p$-groups, Massey products, cohomological dimension, one-relator pro-$p$-groups.}

\section{Statement of results}

The central objects studied in this paper are \f{mild pro-$p$-groups} whose concept is due to J.\ Labute. By definition mild pro-$p$-groups admit a presentation in terms of generators and relators, where the relators satisfy a condition of being free in a maximal possible way with respect to suitably chosen filtrations. Mild pro-$p$-groups are of cohomological dimension $2$ which is one of the main reasons for their importance also from an arithmetic point of view. 

In \cite{JL}, using the theory of mild pro-$p$-groups, J.\ Labute gave the first examples of groups of the form $G_S(p)=\Gal(\Q_S(p)|\Q)$ satisfying $\cd G_S(p)=2$ in the \f{tame case}, i.e.\ where $p\not\in S$. Here $p$ is an odd prime number and $\Q_S(p)$ denotes the maximal pro-$p$-extension of $\Q$ unramified outside $S$. This has led to further extensive contributions in more general situations including arbitrary number fields and the case $p=2$. A.\ Schmidt showed that over a number field the sets $S$ of primes satisfying $\cd G_S(p) = 2$ are cofinal among all finite sets of primes even in the tame case (cf.\ \cite{AS2}, Th.\ 1.1).

These arithmetic results require efficient methods to decide when a given pro-$p$-group is mild. Such a criterion is given by the following

\pagebreak

\vspace{10pt} \noindent
\bf Cup-product criterion:\it

\noindent Let $p$ be prime number and $G$ be a finitely presented pro-$p$-group such that \linebreak $H^2(G,\F_p)\neq 0$. Assume that $H^1(G,\F_p)$ admits a decomposition $H^1(G,\F_p)=U\oplus V$ as $\F_p$-vector space, such that the following holds: 
\begin{itemize}
 \item[\rm (i)] The cup-product $V\otimes V\stackrel{\cup}{\longrightarrow} H^2(G,\F_p)$ is trivial. 
 \item[\rm (ii)] The cup-product $U\otimes V\stackrel{\cup}{\longrightarrow} H^2(G,\F_p)$ is surjective. 
\end{itemize}
Then $G$ is a mild pro-$p$-group, in particular $\cd G=2$ holds.
\rm
\vspace{10pt} 

For $p>2$, this result is a cohomological reformulation of \cite{AS3}, Th.5.5, which in turn follows from a result due to J.\ Labute on so-called \f{strongly free sequences} in Lie algebras. For $p=2$ it has been proven by J.\ Labute and J.\ Miná\v{c} \cite{LM} using mixed Lie algebras and later independently by P.\ Forré \cite{PF2}. Obviously the assumptions of the above criterion imply the surjectivity of the cup-product 
\[
H^1(G,\F_p)\otimes H^1(G,\F_p)\stackrel{\cup}{\longrightarrow} H^2(G,\F_p).
\]
Equivalently, if $G=F/R$ is a minimal presentation of $G$ (i.e.\ $F$ is a free pro-$p$-group and the inflation map $H^1(G,\F_p)\longrightarrow H^1(F,\F_p)$ is an isomorphism) and $r_1,\ldots, r_m\in R$ is a minimal system of defining relations, then the $r_i$ must be linearly independent modulo $F_{(3)}$ where $F_{(n)}$ denotes the \f{Zassenhaus filtration} of $F$. If this is not the case, e.g.\ if $R\subseteq F_{(3)}$, then $G$ can still be mild, but the cup-product does not contain enough information. 

However, if the cup-product is trivial, then inductively there exist well-defined multilinear \f{$n$-fold Massey products}
\[
 \langle \cdot,\ldots, \cdot \rangle_n: H^1(G,\F_p)^n\longrightarrow H^2(G,\F_p),\ n\ge 2,
\]
which can be seen as higher analogs of the cup-product. As noticed by H. Koch, there is a close connection between these higher Massey products and relations lying in higher Zassenhaus filtration steps of $F$.  The exact identities have been proven independently by D.\ Vogel and M.\ Morishita (\cite{DV2} and \cite{MM2}). 

The main result of this paper is to prove the following generalization of the cup-product criterion to arbitrary higher Massey products. We define the \f{Zassenhaus invariant} $\z(G)$ by
\[
 \z(G)=\sup_{n\ge 2}\left\{\langle \cdot,\ldots,\cdot \rangle_k: H^1(G,\F_p)^k\to H^2(G,\F_p)\ \mbox{is trivial for all}\ k <n\right\}.
\]

\begin{theono}[Th.\ref{cohomologicalcrit}]
 Let $p$ be a prime number and $G$ a finitely presented pro-$p$-group with $n=\z(G)<\infty$. Assume that $H^1(G,\F_p)$ admits a decomposition $H^1(G,\F_p)=U\oplus V$ as $\F_p$-vector space such that for some natural number $e$ with $1\le e \le n-1$ the $n$-fold Massey product $\langle\cdot, \ldots, \cdot \rangle_n: H^1(G,\F_p)^n \longrightarrow H^2(G,\F_p)$ satisfies the following conditions:
\begin{itemize}
 \item[\rm (a)] $\langle \xi_1,\ldots, \xi_n \rangle_n =0$ for all tuples $(\xi_1,\ldots, \xi_n)\in H^1(G,\F_p)^n$ such that $\#\{i\ |\ \xi_i\in V\}\ge n-e+1$.
 \item[\rm (b)] $\langle\cdot, \ldots, \cdot \rangle_n$ maps 
\[
U^{\otimes e} \otimes V^{\otimes n-e}
\]
surjectively onto $H^2(G,\F_p)$.
\end{itemize}
Then $G$ is mild (with respect to the Zassenhaus filtration). In particular, $G$ is of cohomological dimension $\cd G=2$.
\end{theono}

Note that we do \it not \rm assume that $p\nmid \z(G)$. Hence, in the case $\z(G)=2$, we obtain a new proof of the cup-product criterion including the case $p=2$. The above theorem allows to construct mild pro-$p$-groups with relations of arbitrary degree with respect to the Zassenhaus filtration.

A concept very closely related to the idea of mild pro-$p$-groups has already been introduced in \cite{JLOneRel} in the case of \f{one-relator pro-$p$-groups}. By an open question due to Serre and Gildenhuys, a one-relator pro-$p$-group $G$ should be of cohomologicial dimension $\cd G=2$ unless the generating relation can be chosen to be a $p$-th power. It turns out that the theory of mild pro-$p$-groups with respect to Zassenhaus filtrations and higher Massey products yields a number of interesting applications in the case of one-relator pro-$p$-groups:

\begin{theono}[Cor.\ref{onerelatorzassenhaus}]
 Let $G$ be a one-relator pro-$p$-group such that the Zassenhaus invariant $\z(G)$ is prime to $p$. Then $G$ is mild with respect to the Zassenhaus filtration. In particular, $\cd G=2$.
\end{theono}

This result is a special case of the more general Theorem \ref{onerelatorgrouptheo}. We deduce that the Zassenhaus invariant of a one-relator pro-$p$-group with cohomological dimension $>2$ is necessarily divisible by $p$. With regard to Serre's question, the following result is of interest:

\begin{theono}[Prop.\ref{onerelatorapprox}]
 Let $G$ be a one-relator pro-$p$-group with generator rank $d=h^1(G)$ and Zassenhaus invariant $\z(G)=p$ and $G=F/(r)$ be a minimal presentation of $G$. Assume that $\cd G > 2$. Then there exists a free basis $x_1,\ldots, x_d$ of $F$ and $y\in F$ such that
\[
 r \equiv x_1^p\modulo F_{(p+1)}.
\]
\end{theono}

Important examples of one-relator pro-$p$-groups are given by the class of \f{Demu\v{s}kin groups}, i.e.\ one-relator pro-$p$-groups $G$ with non-degenerate cup-product $H^1(G,\F_p)\times H^1(G,\F_p)\to H^2(G,\F_p)$. These groups are duality groups of dimension $2$. We introduce the a more general notion of groups of \f{Demu\v{s}kin type} with higher Zassenhaus invariants and prove that they are either finite or mild.

The article is structered as follows: In the second section we investigate the notion of mild pro-$p$-groups with respect to weighted Zassenhaus filtrations. This is basically contained in \cite{JL}, however we obtain more explicit results on the structure of the graded restricted Lie algebras $\gr G$. In the third section we state Anick's criterion for combinatorially free sequences of monomials in non-commutative polynomial algebras and construct a special class of multiplicative monomial orders which will be crucial for the proof of Theorem \ref{cohomologicalcrit}. In section 4 we recall the notion and basic properties of higher Massey products in group cohomology and prove Theorem \ref{cohomologicalcrit}. Furthermore, we give arithmetic examples of mild pro-$2$-groups with Zassenhaus invariant $3$. The fifth section is devoted to the study of one-relator pro-$p$-groups. We investigate the structure of the graded restricted Lie algebras $\gr^\tau F$ for free pro-$p$-groups $F$ which is crucial for the proofs of the main results \ref{onerelatorgrouptheo} and \ref{onerelatorapprox}. In the last section we introduce the notion of groups of Demu\v{s}kin type and apply \ref{cohomologicalcrit} to show that they are either finite or mild.

\bfseries Acknowledgements: \mdseries Parts of the results in this article are contained in the author's PhD thesis. Hearty thanks go to Kay Wingberg for his guidance and great support and to Alexander Schmidt and Denis Vogel for helpful discussions on the subject.

\section{Mild pro-$p$-groups with respect to weighted Zassenhaus filtrations}

Let $p$ be a prime number and $G$ be a pro-$p$-group. By $\Omega_G$ we denote the complete group algebra
\[
 \Omega_G=\grg=\prlim{U} \F_p[G/U]
\]
 with coefficients in $\F_p$, where $U$ runs through the open normal subgroups of $G$. By $I_G\subseteq\Omega_G$ we denote the \f{augmentation ideal} of $G$, i.e.\ the kernel of the canonical \f{augmentation map}
\[
 \begin{tikzpicture}[description/.style={fill=white,inner sep=2pt}, bij/.style={above,sloped,inner sep=1.5pt}, column 1/.style={anchor=base east},column 2/.style={anchor=base west}]
 \matrix (m) [matrix of math nodes, row sep=0em,
 column sep=2.5em, text height=1.5ex, text depth=0.25ex]
 {\Omega_G & \F_p,\\
  g & 1,& g\in G.\\
};
 \path[->>,font=\scriptsize]
 (m-1-1) edge node[auto] {} (m-1-2);
 \path[|->,font=\scriptsize]
 (m-2-1) edge node[auto] {} (m-2-2);
 \end{tikzpicture}
\]

The \f{Zassenhaus filtration} $(G_{(n)})_{n\in\N}$ of $G$ is defined by
\[
 G_{(n)}=\{g\in G\ |\ g-1\in I_G^n\}.
\]
If $G$ is finitely generated as pro-$p$-group, then $G_{(n)}, n\in\N$ is a system of neighborhoods $1\in G$ consisting of open normal subgroups. The Zassenhaus filtration is a \f{$p$-restricted filtration} of $G$, i.e.\ it satisfies
\[
 [G_{(i)},G_{(j)}]\subseteq G_{(i+j)},\quad (G_{(i)})^p\subseteq G_{(pi)}.
\]
Hence, the associated graded object
\[
 \gr G=\bigoplus_{n\ge 1} G_{(n)}/G_{(n+1)}
\]
carries the structure of a \f{graded restricted Lie algebra} over $\F_p$ in the sense of Jacobson (cf.\ \cite{NJ}, Ch.V.7) where the Lie bracket is induced by taking commutators and the $p$-th power operation is induced by taking $p$-th powers in $G$. Before proceeding, we quickly collect some basic facts on ideals and quotients of restricted Lie algebras defined over a field $k$ of characteristic $p$. Homomorphisms and ideals of restricted Lie algebras are defined in the natural way. Every restricted Lie algebra $L$ over $k$ possesses an associative \f{universal enveloping algebra} $U_L$ together with an injective mapping $L\hookrightarrow U_L$ such that the functor $L\longmapsto U_L$ is left adjoint to the canonical functor $\{\mbox{algebras over}\ k\}\longrightarrow \{\mbox{restricted Lie algebras over}\ k\}$, cf.\ \cite{NJ}, Ch.V, Th.12. We need the following fact on universal enveloping algebras of quotients:

\begin{prop}
\label{universalquot}
Let $L$ be a restricted Lie algebra over a field $k$ of characteristic $p$ and let $\mf{r}\subseteq L$ be an ideal. Let $\mf{R}$ denote the left ideal of $U_L$ generated by the image of $\mf{r}$ under the embedding $L\hookrightarrow U_L$. Then $\mf{R}$ is a two-sided ideal and the canonical surjection $\xyalign\xymatrix{L\ar@{->>}[r] & L/\mf{r}}$ induces an exact sequence
\[
  \begin{tikzpicture}[description/.style={fill=white,inner sep=2pt}, bij/.style={above,sloped,inner sep=1.5pt}, column 1/.style={anchor=base east},column 2/.style={anchor=base west}]
 \matrix (m) [matrix of math nodes, row sep=1.5em,
 column sep=2.5em, text height=1.5ex, text depth=0.25ex]
 { 0 & \mf{R} & U_L & U_{L/\mf{r}} & 0.\\
};
 \path[->,font=\scriptsize]
(m-1-1) edge node[auto] {} (m-1-2)
(m-1-2) edge node[auto] {} (m-1-3)
(m-1-3) edge node[auto] {} (m-1-4)
(m-1-4) edge node[auto] {} (m-1-5);
 \end{tikzpicture}
\]
\end{prop}

For the analogous statement for (ordinary) Lie algebras, see \cite{NB}, Ch.I, §2.3, Prop.3. Since the universal enveloping algebra of a restricted Lie algebra satisfies the analogous universal property as for an (ordinary) Lie algebra, the proof carries over immediately. 

For finitely generated free pro-$p$-groups we can define a more general class of Zassenhaus filtrations by associating weights to a fixed system of generators. Let $F$ be the free pro-$p$-group on the set $x=\{x_1,\ldots, x_d\}$. Denoting by $\power$ the algebra of formal power series in the non-commuting indeterminates $X=\{X_1,\ldots, X_d\}$ over $\F_p$, there is a topological isomorphism
 \[
 \begin{tikzpicture}[description/.style={fill=white,inner sep=2pt}, bij/.style={above,sloped,inner sep=1.5pt}, column 1/.style={anchor=base east},column 2/.style={anchor=base west}]
 \matrix (m) [matrix of math nodes, row sep=0em,
 column sep=2.5em, text height=1.5ex, text depth=0.25ex]
 {\Omega_F  & \power,\\
 x_i &1+X_i,\\
};
 \path[->,font=\scriptsize]
 (m-1-1) edge node[bij] {$\sim$} (m-1-2);
 \path[|->,font=\scriptsize]
 (m-2-1) edge node[auto] {} (m-2-2);
 \end{tikzpicture}
\]
mapping the augmentation ideal $I_F$ onto the (two-sided) ideal 
\[
I(X)=\langle X_1, \ldots, X_d\rangle = \bigoplus_{i=1}^d \power \cdot X_i.
\]

Keeping this notation, we can make the following

\begin{defi}
\label{xtaudefinition} 
Let $\tau=(\tau_1,\ldots,\tau_d)$ be a sequence of integers $\tau_i>0$. We define the valuation $\nu_\tau$ of $\power$ by
\[
 \nu_\tau(\sum_{i_1,\ldots, i_k} a_{i_1,\ldots, i_k} X_{i_1}\cdots X_{i_k}) = \inf_{\stackrels{i_1,\ldots, i_k}{a_{i_1,\ldots, i_k}\neq 0}} (\tau_{i_1}+\ldots+\tau_{i_k})
\]
and set $\power^\tau_n=\{h\in\power\ |\ \nu_\tau(h)\ge n\},\ n\ge 0$. The $p$-restricted filtration $(F_ {(\tau,n)})_{n\ge 1}$ of $F$ induced by $\nu_\tau$, i.e.\ the filtration given by the subgroups
\[
 F_{(\tau,n)}=\{f\in F| f-1\in \power^\tau_n\},\ n\ge 1,
\]
where we make the identification $\Omega_F\cong\power$ as above, is called \f{Zassenhaus $(x,\tau)$-filtration} of $F$. We write $\gr^\tau F = \bigoplus_{n\ge 1} F_{(\tau,n)}/F_{(\tau,n+1)}$ for the associated graded restricted Lie algebra.
\end{defi}                            

Note that unless $\tau_1=\tau_2=\dots=\tau_d$, the Zassenhaus $(x,\tau)$-filtration depends on the choice of the generating set for $F$. Also note that if $\tau_i=1,\ i=1,\ldots, d$, we obtain the usual Zassenhaus filtration of $F$. 

We denote by $\gr^\tau\power$ the graded $\F_p$-algebra with respect to the filtration $(\power^\tau_n)$, i.e.
\[
 \gr^\tau\power=\bigoplus_{n\ge 0} \power^\tau_n/\power^\tau_{n+1}.
\]
If $\polyx^\tau$ denotes the free associative algebra over $\F_p$ on $X$ (i.e.\ the polynomial algebra in the non-commuting variables $X=\{X_1,\ldots, X_d\}$) endowed with the grading given by $\deg X_i=\tau_i$, there is a canonical isomorphism
\[
 \begin{tikzpicture}[description/.style={fill=white,inner sep=2pt}, bij/.style={above,sloped,inner sep=1.5pt}, column 1/.style={anchor=base east},column 2/.style={anchor=base west}]
 \matrix (m) [matrix of math nodes, row sep=0em,
 column sep=2.5em, text height=1.5ex, text depth=0.25ex]
 {\gr^\tau\power  & \polyx^\tau\\
 \xi_i &X_i\\
};
 \path[->,font=\scriptsize]
 (m-1-1) edge node[bij] {$\sim$} (m-1-2);
 \path[|->,font=\scriptsize]
 (m-2-1) edge node[auto] {} (m-2-2);
 \end{tikzpicture}
\]
of graded $\F_p$-algebras, where $\xi_i\in \power^\tau_{\tau_i}/\power^\tau_{\tau_i+1}$ is the initial form of $X_i\in \power$. Making the identification $\Omega_F\cong \power$, the map $F\hookrightarrow\Omega_F$, $f\mapsto f-1$ induces an embedding of graded restricted Lie algebras
\[
\begin{tikzpicture}[description/.style={fill=white,inner sep=2pt}, bij/.style={above,sloped,inner sep=1.5pt}]
\matrix (m) [matrix of math nodes, row sep=3em,
column sep=2.5em, text height=1.5ex, text depth=0.25ex]
{ \phi_F: \gr^\tau F & \polyx^\tau,\\};
\path[right hook->,font=\scriptsize]
(m-1-1) edge node[auto] {} (m-1-2);
\end{tikzpicture}
\]
mapping the initial form of $x_i$ to $X_i$.

\begin{rema}
 Via $\phi_F$, $\gr^\tau F$ is mapped onto the restricted Lie subalgebra of $\polyx^\tau$ generated by $X_1, \ldots, X_d$, which is the \it free \rm restricted $\F_p$-Lie algebra on the set $X$. Furthermore, $\polyx^\tau$ is its universal enveloping algebra. This yields explicit bases for the $\F_p$-vector spaces $F_{(\tau,n)}/F_{(\tau,n+1)}$. However, we will not make use of this fact until section 5, cf.\ \ref{grftheo}.
\end{rema}

In order to define mild pro-$p$-groups, we need to notion of so-called \f{strongly free sequences} in the free algebra $\polyx^\tau$. We fix the $(X,\tau)$-grading for some \linebreak $\tau=(\tau_1,\ldots, \tau_d),\ \tau_i\ge 1$, in particular, \it homogeneous \rm elements will always be homogeneous with respect to this grading.

\begin{defi}
\label{poincareseriesdefinition}
Let $H$ be a \f{locally finite} graded algebra (Lie algebra) over an arbitrary field $k$, i.e.\ $H=\oplus_{n=0}^\infty H_n$ is a graded $k$-algebra (Lie algebra) such that $H_n$ is of finite dimension as $k$-vector space for all $n$. Then the \f{Poincaré series} (or \f{Hilbert series}) $H(t)\in\Z\llbracket t \rrbracket$ of $H$ is the formal power series
\[
 H(t) = \sum_{n=0}^\infty (\dim_k H_n) t^n.
\]
 It is additive in exact sequences, i.e.\ if
\[
\begin{tikzpicture}[description/.style={fill=white,inner sep=2pt}, bij/.style={below,sloped,inner sep=1.5pt}]
\matrix (m) [matrix of math nodes, row sep=1.5em,
column sep=2.5em, text height=1.5ex, text depth=0.25ex]
{  0 & H' & H & H'' & 0\\};
\path[->,font=\scriptsize]
(m-1-1) edge node[auto] {} (m-1-2)
(m-1-2) edge node[auto] {} (m-1-3)
(m-1-3) edge node[auto] {} (m-1-4)
(m-1-4) edge node[auto] {} (m-1-5);
\end{tikzpicture}
\]
is an exact sequence of graded algebras (Lie algebras) over $k$, then
\[
 H(t)=H'(t)+H''(t).
\]
\end{defi}

We define a total ordering on $\Z\llbracket t\rrbracket$ by setting $f(t)>g(t)$ if the first non-zero coefficient of $f(t)-g(t)$ is positive. This ordering satisfies the usual compatibility properties with respect to addition and multiplication in $\Z\llbracket t\rrbracket$.

For the sake of brevity, in all that follows we set $A=\polyx^\tau$ and denote by $I_A=\langle X_1,\ldots, X_d\rangle=\{f\in\polyx^\tau|\ \deg^\tau(f)>0\}$ the augmentation ideal of $A$.

\pagebreak

\begin{lemm}
\label{poincareestimate}
 Let $\rho_1,\ldots, \rho_m\in I_A$ be homogeneous elements of degree $\sigma_i=\deg^\tau \rho_i,\ i=1,\ldots, m$,
\[
 \mc{R}= (\rho_1,\ldots, \rho_m)
\]
be the two-sided ideal of $A$ generated by the $\rho_i$ and $B=A/\mc{R}$ be the quotient algebra. We endow $B$ with the natural induced grading. Then the Poincaré series $B(t)$ satisfies
\[
 B(t)\ge \frac{A(t)}{1+A(t)(t^{\sigma_1}+\ldots+t^{\sigma_m})} = \frac{1}{1-(t^{\tau_1}+\ldots+t^{\tau_d}) + (t^{\sigma_1}+\ldots+t^{\sigma_m})}.
\]
\end{lemm}
\begin{proof}
 We follow the arguments given by J.\ Labute in \cite{JL}, Prop.3.2. Since $\mc{R}\subseteq I_A$, the augmentation sequence
\[
 \begin{tikzpicture}[description/.style={fill=white,inner sep=2pt}, bij/.style={above,sloped,inner sep=1.5pt}]
\matrix (m) [matrix of math nodes, row sep=3em,
column sep=2.5em, text height=1.5ex, text depth=0.25ex]
{ 0 & I_A & A & \F_p & 0\\};
\path[->,font=\scriptsize]
(m-1-1) edge node[auto] {} (m-1-2)
(m-1-2) edge node[auto] {} (m-1-3)
(m-1-3) edge node[auto] {} (m-1-4)
(m-1-4) edge node[auto] {} (m-1-5);
\end{tikzpicture}
\]
gives rise to the exact sequence
\[
 \begin{tikzpicture}[description/.style={fill=white,inner sep=2pt}, bij/.style={above,sloped,inner sep=1.5pt}]
\matrix (m) [matrix of math nodes, row sep=3em,
column sep=2.5em, text height=1.5ex, text depth=0.25ex]
{ 0 & \mc{R}/\mc{R}I_A & I_A/\mc{R}I_A & B & \F_p & 0.\\};
\path[->,font=\scriptsize]
(m-1-1) edge node[auto] {} (m-1-2)
(m-1-2) edge node[auto] {} (m-1-3)
(m-1-3) edge node[auto] {} (m-1-4)
(m-1-4) edge node[auto] {} (m-1-5)
(m-1-5) edge node[auto] {} (m-1-6);
\end{tikzpicture}
\]
By left multiplication, the quotient algebra $I_A/\mc{R}I_A$ is a free $B$-module over the images of $X_1, \ldots, X_d$. Furthermore, $\mc{R}/\mc{R}I_A$ is a left $B$-module generated by $\rho_1, \ldots, \rho_m$ using the fact that $R A = R (\F_p \oplus I_A) =  R\F_p \oplus R I_A$ where $R$ denotes the $\F_p$-span of $\rho_1,\ldots, \rho_m$ in $A$. Hence, one has a surjective map of graded $\F_p$-vector spaces
\[
 \begin{tikzpicture}[description/.style={fill=white,inner sep=2pt}, bij/.style={above,sloped,inner sep=1.5pt}]
\matrix (m) [matrix of math nodes, row sep=3em,
column sep=2.5em, text height=1.5ex, text depth=0.25ex]
{ \bigoplus_{i=1}^m B[\sigma_i] & \mc{R}/\mc{R}I_A\\};
\path[->>,font=\scriptsize]
(m-1-1) edge node[auto] {} (m-1-2);
\end{tikzpicture}
\]
where for $l\in\N_0$ by $B[l]$ we denote the $\F_p$-vector space $B$ with grading shifted by $l$, i.e.\ $B[l](t) = t^l B(t)$.
Summing up, we get an exact sequence of graded $\F_p$-vector spaces
\[
 \begin{tikzpicture}[description/.style={fill=white,inner sep=2pt}, bij/.style={above,sloped,inner sep=1.5pt}]
\matrix (m) [matrix of math nodes, row sep=3em,
column sep=2.5em, text height=1.5ex, text depth=0.25ex]
{ \bigoplus_{i=1}^m B[\sigma_i] & \bigoplus_{i=1}^d B[\tau_i] & B & \F_p & 0.\\};
\path[->,font=\scriptsize]
(m-1-1) edge node[auto] {} (m-1-2)
(m-1-2) edge node[auto] {} (m-1-3)
(m-1-3) edge node[auto] {} (m-1-4)
(m-1-4) edge node[auto] {} (m-1-5);
\end{tikzpicture}
\]
 Taking Poincaré series, we obtain
\[
 -(t^{\sigma_1}+\ldots + t^{\sigma_m})B(t) + (t^{\tau_1}+\ldots+t^{\tau_d}) B(t) - B(t) + 1 \le 0 
\]
and solving for $B(t)$, this gives the desired inequality.
\end{proof}

\begin{defi}
\label{stronglyfreedefi}
Keeping the notation of \ref{poincareestimate}, the sequence of homogeneous elements $\rho=\{\rho_1,\ldots, \rho_m\}\subset I_A$ is called \f{strongly free} if the inequality of Poincaré series given in \ref{poincareestimate} is an equality, i.e.\ if
\begin{eqnarray*}
 B(t) = \frac{1}{1-(t^{\tau_1}+\ldots+t^{\tau_d}) + (t^{\sigma_1}+\ldots+t^{\sigma_m})}.
\end{eqnarray*}
Furthermore, we also say that the empty sequence $\rho=\emptyset$ is strongly free, in which case $B=A$ and hence
\[
 B(t)= \frac{1}{1-(t^{\tau_1}+\ldots+t^{\tau_d})}.
\]
\end{defi}
\pagebreak

\begin{rema}
\label{stronglyfreedefiremark}\quad
\begin{itemize}
 \item[\rm (i)] Note that in the above definition, by a slight abuse of notation, $\rho=\{\rho_1,\ldots, \rho_m\}$ is considered as the \it sequence \rm (and not just as the \it subset\rm) consisting of $\rho_1,\ldots, \rho_m$. For instance, the sequence $\rho=\{X_1\}$ is strongly free, whereas the sequence $\rho=\{X_1,X_1\}$ is not. However, obviously the strong freeness of $\rho$ does not depend on the ordering of the $\rho_i$.
 \item[\rm (ii)] In general the series on the right hand side of the inequality in \ref{stronglyfreedefi} does not have positive coefficients which is of course a necessary condition for the existence of a strongly free sequence with given degrees.
\end{itemize}
\end{rema}

\begin{prop}
\label{strfreeconditions}
 Let $\rho_1,\ldots, \rho_m\subset I_A$ be homogeneous elements. Then the following statements are equivalent:
\begin{itemize}
 \item[\rm (i)] $\rho_1, \ldots, \rho_m$ is a strongly free sequence.
 \item[\rm (ii)] The left $B$-module $\mc{R}/\mc{R}I_A$ is free over the images of $\rho_1, \ldots, \rho_m$.
 \item[\rm (iii)] The subalgebra $\mf{R}\subseteq A$ generated by $\rho_1,\ldots, \rho_m$ is free over $\rho_1,\ldots, \rho_m$ and there is an isomorphism of graded $\F_p$-vector spaces
\[
\begin{tikzpicture}[description/.style={fill=white,inner sep=2pt}, bij/.style={above,sloped,inner sep=1.5pt}]
\matrix (m) [matrix of math nodes, row sep=3em,
column sep=2.5em, text height=1.5ex, text depth=0.25ex]
{ \mf{R}\amalg (A/\mc{R}) & A\\};
\path[->,font=\scriptsize]
(m-1-1) edge node[bij] {$\sim$} (m-1-2);
\end{tikzpicture}
\]
where $\amalg$ denotes the coproduct in the category of connected, locally finite graded algebras over $\F_p$.
\end{itemize}
\end{prop}
\begin{proof}
 The equivalence (i)$\Leftrightarrow$(ii) follows directly from the proof of \ref{poincareestimate}. For the equivalence (i)$\Leftrightarrow$(iii) see \cite{DA}, Th.2.6.
\end{proof}

Before proceeding, we give two examples in the case $\tau=(1,\ldots, 1)$:

\begin{exam}\quad
\label{strfreeexample}
\begin{itemize} 
 \item[\rm (i)]
 Let $d\ge 4$ and
\[
 \rho_1=[X_1,X_2],\ \rho_2=[X_2,X_3],\ldots,\ \rho_{d-1}=[X_{d-1},X_d],\ \rho_d=[X_d,X_1].
\]
Then the sequence $\rho_1, \ldots, \rho_d$ is strongly free. In fact, this can be shown using Labute's criterion on \f{non-singular circuits}, cf.\ \cite{JL} or \f{Anick's criterion}, cf.\ \ref{anickscrit}.
 \item[\rm (ii)] Let $d\ge 3$ and 
\[
  \rho_1=[X_1,X_2],\ \rho_2=[X_2,X_3],\ \rho_3=[X_3,X_1].
\]
The sequence $\rho_1,\ldots, \rho_3$ is \it not \rm strongly free. In fact, there can be no strongly free sequence of $3$ homogeneous polynomials of degree $2$ in $3$ variables, since the series
\[
 \frac{1}{1-3t+3t^2} = 1+3t+6t^2+9t^3+9t^4+0t^5-27t^6 + \mc{O}(t^7)
\]
contains strictly negative coefficients. 
\end{itemize}
\end{exam}

We will now study the notion of \f{mild} pro-$p$-groups. In the following, for a pro-$p$-group $G$ we set 
\[
H^i(G)=H^i(G,\F_p)\ \mbox{and}\ h^i(G)=\dimfp H^i(G),\ i\ge 0
\]
A pro-$p$-group $G$ is called \f{finitely presented} if $h^1(G),h^2(G)<\infty$. A presentation
 \[
\begin{tikzpicture}[description/.style={fill=white,inner sep=2pt}, bij/.style={below,sloped,inner sep=1.5pt}]
\matrix (m) [matrix of math nodes, row sep=1.5em,
column sep=2.5em, text height=1.5ex, text depth=0.25ex]
{  1 & R & F & G & 1\\};
\path[->,font=\scriptsize]
(m-1-1) edge node[auto] {} (m-1-2)
(m-1-2) edge node[auto] {} (m-1-3)
(m-1-3) edge node[auto] {} (m-1-4)
(m-1-4) edge node[auto] {} (m-1-5);
\end{tikzpicture}
\]
of $G$ by a free pro-$p$-group $F$ is called \f{minimal} if the inflation map 
\[
\begin{tikzpicture}[description/.style={fill=white,inner sep=2pt}, bij/.style={below,sloped,inner sep=1.5pt}]
\matrix (m) [matrix of math nodes, row sep=1.5em,
column sep=2.5em, text height=1.5ex, text depth=0.25ex]
{  \infl:\ H^1(G) & H^1(F)\\};
\path[right hook->,font=\scriptsize]
(m-1-1) edge node[auto] {} (m-1-2);
\end{tikzpicture}
\]
is an isomorphism. If $F$ is the free pro-$p$-group on $x=\{x_1,\ldots, x_d\}$ endowed with the Zassenhaus $(x,\tau)$-filtration (where as before $\tau=(\tau_1,\ldots, \tau_d)$ for some natural numbers $\tau_i\ge 1$), then the latter is equivalent to $R\subseteq F_{(\tau,2)}$. Furthermore, if $R$ is generated by $r_1,\ldots, r_m\in F_{(\tau,2)}$ as closed normal subgroup of $F$, we also write
\[
 G=\langle x_1,\ldots, x_d\ |\ r_1,\ldots, r_m\rangle.
\]

\begin{defi}\quad
\label{milddefi}
\begin{itemize}
 \item[\rm (i)] Let $F$ be the free pro-$p$-group on $x=\{x_1,\ldots, x_d\}$ endowed with the $(x,\tau)$-filtration for some $\tau=(\tau_1,\ldots, \tau_d)$ and $r_1,\ldots, r_m\in F_{(\tau,2)}=F^{(2)}$. Let $X=\{X_1,\ldots, X_d\}$ and suppose that the sequence $\rho=\{\rho_1,\ldots, \rho_m\}$ of the initial forms $\rho_i$ of $r_i$
\[
 \rho_1,\ldots, \rho_m \in \gr^\tau F \subset \polyx^\tau
\]
is strongly free. Then $G=\langle x_1,\ldots, x_d\ |\ r_1,\ldots, r_m\rangle$ is called \f{strongly free presentation} with respect to the Zassenhaus $(x,\tau)$-filtration of the pro-$p$-group $G$.
 \item[\rm (ii)] A finitely generated pro-$p$-group $G$ with $d=h^1(G)$ is called \f{mild} with respect to the Zassenhaus $(x,\tau)$-filtration if it possesses a strongly free presentation with respect to the Zassenhaus $(x,\tau)$-filtration.
\end{itemize}
\end{defi}

The following theorem collects the main properties of mild pro-$p$-groups. 

\begin{theo}
\label{mainmildtheo}
Let $F$ be the free pro-$p$-group on $x=\{x_1,\ldots, x_d\}$ endowed with the $(x,\tau)$-filtration for some $\tau=(\tau_1,\ldots, \tau_d)$. Let $G$ be a mild pro-$p$-group such that $G=F/R=\langle x_1,\ldots, x_d\ |\ r_1,\ldots, r_m\rangle, \ d,m\ge 1$ is a strongly free presentation with respect to the Zassenhaus $(x,\tau)$-filtration where $R\subseteq F_{(\tau,2)}$ denotes the closed normal subgroup generated by $r_i,\ i=1,\ldots, m$. Set $\sigma_i=\deg^\tau r_i,\ i=1,\ldots, m$. Furthermore, let $(G_{(\tau,n)})_{n\in\N}$ denote the filtration of $G$ induced by the projection $\xyalign\xymatrix@C=18pt{F\ar@{->>}[r] & G}$ and $\gr^\tau G=\bigoplus_{i\ge 1} G_{(\tau,i)}/G_{(\tau,i+1)}$ the associated graded restricted $\F_p$-Lie algebra. Then the following holds:
\begin{itemize}
 \item[\rm (i)] We have $\cd G = 2$ and the relation rank $h^2(G)$ of $G$ is equal to $m$.
 \item[\rm (ii)] The $\Omega_G$-module $R/R^p[R,R]$ is free over the images of $r_1,\ldots, r_m$.
 \item[\rm (iii)] We have $\gr^\tau G=\gr^\tau F / (\rho_1,\ldots, \rho_m)$ where $(\rho_1,\ldots, \rho_m)$ denotes the ideal of the restricted Lie algebra $\gr^\tau F$ generated by the initial forms $\rho_i$ of $r_i$.
 \item[\rm (iv)] The universal enveloping algebra $U_{\gr^\tau G}$ of $\gr^\tau G$ is the graded algebra associated to the filtration on $\Omega_G$ induced by the $(x,\tau)$-filtration on $\Omega_F$ and its Poincaré series satisfies
\[
 U_{\gr^\tau G}(t) = \frac{1}{1-(t^{\tau_1}+\ldots+t^{\tau_d}) + (t^{\sigma_1}+\ldots+t^{\sigma_m})}.
\]
 \item[\rm (v)] If $m\neq d-1$, then $G$ is not $p$-adic analytic.
 \end{itemize}
\end{theo}

\begin{remas}\quad
\begin{itemize}
 \item[\rm (i)] If the initial forms $\rho_1,\ldots, \rho_m$ lie in the free Lie subalgebra of $\polyx^\tau$ generated by $X$ (i.e.\ if they are Lie polynomials in the variables \linebreak $X_1,\ldots, X_d$), the above statements are contained in \cite{JL}, Th.5.1. The generalization to arbitrary sequences is crucial for the results on higher Massey products as well as for applications to arithmetically defined pro-$p$-groups. A different proof of the fact that mild pro-$p$-groups are of cohomological dimension $2$ involving (associative) $\F_p$-algebras only is given in \cite{PF2}.
 \item[\rm (ii)] The fact that one can check mildness with respect to various Zassenhaus $(x,\tau)$-filtrations is an obvious but useful property. For instance, consider the pro-$p$-group
\[
 G=\langle x_1, x_2\ |\ x_1^2 x_2^4\rangle,
\]
which is not mild with respect to the Zassenhaus filtration. However, it is mild with respect to the Zassenhaus $(x,\tau)$-filtration where $\tau_1=2,\tau_2=1$. 
 \end{itemize}
\end{remas}

In the following proof we can adopt the main ideas of the proof of \cite{JL}, Th.4.1. There similar results are obtained in the case of strongly free sequences with respect to (generalized) $p$-central series. The main difference lies in the fact that we have to deal with graded (associative) $\F_p$-algebras, whereas Labute's notion of strong freeness is in terms of free Lie algebras over the polynomial ring $\F_p[\pi]$.

\begin{proof}[\it Proof of Theorem \ref{mainmildtheo}]
For the sake of simplicity and ignoring our notational conventions, we omit the indices $\tau$ during the proof, hence we set $F_{(n)}=F_{(\tau,n)},$ $\gr_n F=(\gr^\tau(F))_n = F_{(\tau,n)}/F_{(\tau,n+1)}$ etc. By $(R_{(n)})_{n\in\N}$ we denote the induced filtration on $R$, i.e.\ $R_{(n)}=F_{(n)}\cap R$. Furthermore, we set $M=R/R^p[R,R]$ and set $M_{(n)}=\pi(R_{(n)})$ where $\pi$ is the canonical surjection $\xyalign\xymatrix@C=18pt{\pi: R\ar@{->>}[r] & M.}$ Recall that $I_F$ is the augmentation ideal of $\Omega_F$ and we that make the identification $A=\F_p\langle X_1,\ldots, X_d\rangle\cong \gr\Omega_F$. Let $\hat{\mc{R}}\subseteq \Omega_F$ denote the closed two-sided ideal generated by $r-1,\ r\in R$. The canonical map $R\longrightarrow \hat{\mc{R}},\ r\longmapsto r-1$ induces an isomorphism of graded $\Omega_G$-modules
 \begin{eqnarray*}
\label{phiiso}
\begin{tikzpicture}[description/.style={fill=white,inner sep=2pt}, bij/.style={above,sloped,inner sep=1.5pt}, column 1/.style={anchor=base east},column 2/.style={anchor=base west}]
\matrix (m) [matrix of math nodes, row sep=1.5em,
column sep=2.5em, text height=1.5ex, text depth=0.25ex]
{ \phi:\ M & \hat{\mc{R}} / \hat{\mc{R}} I_{\hat{A}}\\};
\path[->,font=\scriptsize]
(m-1-1) edge node[bij] {$\sim$} (m-1-2);
\end{tikzpicture}
\end{eqnarray*}
cf.\ \cite{AB}, proof of Th.5.2.
Let $\mc{R}$ denote the two-sided ideal of $A$ generated by the initial forms of $\rho_i$ of $r_i,\ i=1,\ldots, m$. By definition, we have the inclusion $\mc{R}\hookrightarrow \gr \hat{\mc{R}}$. We endow the ideals $\hat{\mc{R}}$ and $\hat{\mc{R}}I_F$ with the filtrations induced by the filtration on $\Omega_F$. We obtain a homomorphism
\[
 \psi: \mc{R}/\mc{R}I_A\longrightarrow \gr \hat{\mc{R}}/\gr \hat{\mc{R}}I_F = \gr \hat{\mc{R}}/\hat{\mc{R}}I_F
\]
of graded $\F_p$-algebras. We claim that $\psi$ is an isomorphism. This is equivalent to showing that the composite 
\[
 \begin{tikzpicture}[description/.style={fill=white,inner sep=2pt}, bij/.style={above,sloped,inner sep=1.5pt}, column 1/.style={anchor=base east},column 2/.style={anchor=base west}]
\matrix (m) [matrix of math nodes, row sep=1.5em,
column sep=2.5em, text height=1.5ex, text depth=0.25ex]
{ \xi:\ \mc{R}/\mc{R}I_A & \gr \hat{\mc{R}}/\hat{\mc{R}}I_{\hat{A}} & \gr M\\};
\path[->,font=\scriptsize]
(m-1-1) edge node[auto] {} (m-1-2)
(m-1-2) edge node[bij] {$\sim$} (m-1-3);
\end{tikzpicture}
\]
of $\psi$ with the isomorphism $\gr \phi^{-1}$ is again an isomorphism. We prove this by induction on degrees. Let $k\in\N$ and assume that $\xi$ is an isomorphism in degrees $< k$ (if $k=1$, this is clearly fulfilled, since the degree zero subspaces are trivial).
\vspace{10pt}

Injectivity of $\xi$ in degree $k$: First deduce that $\mc{R}_n = \gr_n \hat{\mc{R}}=\hat{\mc{R}}_n/\hat{\mc{R}}_{n+1}$ for all $n<k$. In fact, this follows immediately by induction on degrees from the surjectivity of $\xi$ (and hence of $\psi$) in degrees $<k$. Let $\chi\in \gr_n \hat{\mc{R}} I_F$. Then $\chi$ is the initial form of some element $\sum_{i=1}^l x_i y_i,\ x_i\in\hat{\mc{R}},\ y_i\in I_F$ such that $\deg x_i + \deg y_i = k$ for all $i$. In particular, $\deg x_i,\deg y_i < k$. Using $\mc{R}_n= \gr_n \hat{\mc{R}},\ I_A = \gr_n I_{\hat{A}},\ n<k$, we conclude that $(\mc{R} I_A)_k = \gr_k \hat{\mc{R}} I_F$ and consequently $\psi$ (and hence $\xi$) is injective in degree $k$.
\vspace{10pt}

Surjectivity of $\xi$ in degree $k$: Let $\beta$ be a non-zero element in $\gr_k M$ and choose $b\in M_k=(R/R^p[R,R])_k$ whose initial form is $\beta$. Denoting by $\ol{r}_i$ the image of $r_i$ in $M$, then $\ol{r}_1,\ldots, \ol{r}_m$ generate $M$ as $\Omega_G$-module and we may therefore write
\[
 b = g_1 \ol{r}_1 + \ldots + g_m \ol{r}_m
\]
 with some $g_i\in \Omega_G$. Set $\omega_i=\deg g_i,\ i=1,\ldots, m$ (where we endow $\Omega_G$ with the filtration induced by the $(x,\tau)$-filtration on $\Omega_F$). By definition of the operation of $\Omega_G$ on $M$, we have $(\Omega_G)_i M_j \subseteq M_{i+j}$ and hence we may assume that in the above sum we have $g_i=0$ or $\omega_i+\sigma_i \le k$ for all $i=1,\ldots, m$. Let
\[
 k'=\min_{\stackrels{i=1,\ldots, m,}{g_i\neq 0}} (\omega_i+\sigma_i).
\]
 We claim that $k'=k$. To this end, assume $k'< k$ and define the set $I$ by $I=\{i=1,\ldots, m\ |\ \omega_i + \sigma_i=k'\}$. Let $\ol{\rho}_i$ denote the image of $\rho_i$ in $\mc{R}/\mc{R} I_A$ and set
\[
 \varrho = \sum_{i\in I} u_i \ol{\rho}_i \in \mc{R}/\mc{R} I_A
\]
where $u_i\in A/\mc{R}$ is a preimage of the initial form $\ol{g}_i$ of $g_i$ under the canonical projection
 \[
  \begin{tikzpicture}[description/.style={fill=white,inner sep=2pt}, bij/.style={above,sloped,inner sep=1.5pt}]
\matrix (m) [matrix of math nodes, row sep=1.5em,
column sep=2.5em, text height=1.5ex, text depth=0.25ex]
{ A/\mc{R}&\gr \Omega_G=A / \gr \hat{\mc{R}}.\\
};
\path[->>,font=\scriptsize]
(m-1-1) edge node[auto] {} (m-1-2);
\end{tikzpicture}
\]
By definition $\varrho$ is mapped via $\xi$ to the image of $\sum_{i\in I} g_i \ol{r}_i$ in $\gr_{k'} M=$ \linebreak $M_{k'}/M_{k'+1}$. Since by assumption $\deg b = k > k'$, we have $\xi(\varrho)=0$. By the induction hypothesis, this implies $\varrho = 0$. Using the equivalent characterizations of strong freeness of the sequence $\rho_1,\ldots, \rho_m$ given in \ref{strfreeconditions}, this implies $u_i=0$ for all $i\in I$, which yields a contradiction since $\ol{g}_i\neq 0,\ i\in I$. Hence, we have $k'=k$, i.e.\ $\omega_i+\sigma_i= k$ for all $i$ such that $g_i\neq 0$ and consequently $\beta$ lies in the image of $\xi$. This finishes the induction, i.e.\ $\xi$ and hence $\psi$ are isomorphisms.

By the definition of the filtrations on $R$ and $G$, $\gr R$ is an ideal of the restricted Lie algebra $\gr F$ and we have $\gr G = \gr F / \gr R$. Let $\mc{R}'$ denote the (two-sided) ideal of $A$ generated by the image of $\gr R$ under the inclusion $\gr F\hookrightarrow A=\gr \Omega_F=U_{\gr F}$. Then we have the inclusions
\[
 \mc{R} \subseteq \mc{R}' \subseteq \gr \hat{\mc{R}} \subseteq A.
\]
But as we have already remarked, the surjectivity of $\psi$ implies $\mc{R}=\gr \hat{\mc{R}}$ and consequently $\mc{R}' = \gr \hat{\mc{R}}$. By \ref{universalquot}, the universal enveloping algebra of $\gr G$ is given by
\[
 U_{\gr G} = U_{\gr F}/ \mc{R}' = A/\mc{R} = A / \gr \hat{\mc{R}} = \gr \Omega_G.
\]
From this, it follows immediately that
\[
 U_{\gr G}(t) = (A/\mc{R})(t)=\frac{1}{1-(t^{\tau_1}+\ldots+t^{\tau_d}) + (t^{\sigma_1}+\ldots+t^{\sigma_m})},
\]
since the sequence $\rho_1,\ldots, \rho_m$ is strongly free. This shows (iv). In order to prove statement (iii), let $\mf{r}=(\rho_1,\ldots, \rho_m)\subseteq \gr F$ denote the ideal of the restricted Lie algebra $\gr F$ generated by $\rho_1,\ldots, \rho_m$ and consider the exact sequence
\[
  \begin{tikzpicture}[description/.style={fill=white,inner sep=2pt}, bij/.style={above,sloped,inner sep=1.5pt}]
\matrix (m) [matrix of math nodes, row sep=1.5em,
column sep=2.5em, text height=1.5ex, text depth=0.25ex]
{0 & \gr R / \mf{r} \ & \gr F/\mf{r} & \gr G & 0.\\
};
\path[->,font=\scriptsize]
(m-1-1) edge node[auto] {} (m-1-2)
(m-1-2) edge node[auto] {} (m-1-3)
(m-1-3) edge node[auto] {} (m-1-4)
(m-1-4) edge node[auto] {} (m-1-5);
\end{tikzpicture}
\]
Passing to the universal enveloping algebras and using \ref{universalquot}, we obtain the commutative exact diagram
\[
\begin{tikzpicture}[description/.style={fill=white,inner sep=2pt}, bij/.style={above,sloped,inner sep=1.5pt}]
\matrix (m) [matrix of math nodes, row sep=1.5em,
column sep=2.5em, text height=1.5ex, text depth=0.25ex]
{   &                                  & 0 & 0      \\ 
     &                                 & \mc{R} & \gr \hat{\mc{R}}      \\
     &                                 & U_{\gr F} &  \gr \Omega_F          \\
    0 & \langle \gr R / \mf{r}\rangle \ & U_{\gr F/\mf{r}} & U_{\gr G} & 0\\
&& 0 & 0\\
};
\path[-,font=\scriptsize]
(m-2-3) edge[double,double distance=1.5pt] node[auto] {} (m-2-4)
(m-3-3) edge[double,double distance=1.5pt] node[auto] {} (m-3-4);
\path[->,font=\scriptsize]
(m-4-1) edge node[auto] {} (m-4-2)
(m-4-2) edge node[auto] {} (m-4-3)
(m-4-3) edge node[auto] {} (m-4-4)
(m-4-4) edge node[auto] {} (m-4-5)
(m-1-3) edge node[auto] {} (m-2-3)
(m-2-3) edge node[auto] {} (m-3-3)
(m-3-3) edge node[auto] {} (m-4-3)
(m-4-3) edge node[auto] {} (m-5-3)
(m-1-4) edge node[auto] {} (m-2-4)
(m-2-4) edge node[auto] {} (m-3-4)
(m-3-4) edge node[auto] {} (m-4-4)
(m-4-4) edge node[auto] {} (m-5-4);
\end{tikzpicture}
\]
where $\langle \gr R / \mf{r}\rangle \subseteq U_{\gr F/\mf{r}}$ denotes the (two-sided) ideal generated by the image of $\gr R / \mf{r}$. However, the map $U_{\gr F/\mf{r}} \longrightarrow U_{\gr G}$ being an isomorphism, it follows that $\langle \gr R / \mf{r}\rangle= 0$ and since $\gr F/\mf{r}$ is mapped injectively into its enveloping algebra, it follows that $\gr R = \mf{r}$, showing (iii).

Under the identification $A/\mc{R} = \gr \Omega_G$, the isomorphism $\mc{R}/\mc{R}I_A\cong \gr M$ is an isomorphism of $A/\mc{R}$-modules. Therefore, again using \ref{strfreeconditions}, the $\gr \Omega_G$-module $\gr M$ is free over the initial forms of the images of the $r_i$. Therefore, $M$ is a free $\Omega_G$-module over the images of the $r_i$, cf.\ \cite{ML}, Ch.V, Cor.2.1.1.3. This proves (ii). By \cite{HK}, Th.7.7 the latter is equivalent to $\cd G\le 2$. Since
\[
 h^2(G) = \dimfp H^2(G) = \dimfp H^1(R)^F = \dimfp \Hom (M, \F_p)^F = m\ge 1,
\]
we obtain (i).

Finally, since the sequence $\rho_1,\ldots, \rho_m$ is strongly free, by \cite{DA}, Lemma 3.4 it follows that the polynomial
\[
 (\gr \Omega_G)(t)^{-1} = (A/\mc{R})(t)^{-1} = 1-(t^{\tau_1}+\ldots+t^{\tau_d}) + (t^{\sigma_1}+\ldots+t^{\sigma_m})
\]
has a root in the interval $(0,1]$. If $m\neq d-1$, there is a root in the open interval $(0,1)$. By a result of M.\ Lazard (cf.\ \cite{ML}, App.3, Cor.3.12) it follows that $G$ is not $p$-adic analytic, which shows (v) and concludes the proof.
\end{proof}

\section{Multiplicative monomial orders}

As in the previous section, let $A=\polyx^\tau$ be endowed with the $(X,\tau)$-grading. For a given sequence $\rho_1,\ldots, \rho_m\in I_A$ of homogeneous polynomials, being strongly free is a condition on the (infinitely many) coefficients of the Poincaré series $A/(\rho_1,\ldots, \rho_m)(t)$ which in general is not amenable to straightforward computations by hand. Hence, one needs powerful sufficient criteria to detect strong freeness.

In the special case where $\rho_1,\ldots, \rho_m$ are \it monomials\rm, by a result due to D.\ Anick, the sequence $\rho_1,\ldots, \rho_m$ is strongly free if and only if it is \f{combinatorially free} in the sense of the following

\begin{defi}
\label{combinatoriallyfreedefi}
Let $\rho=\{\rho_1,\ldots,\rho_m\}$ be a sequence of monomials in $I_A$ (i.e.\ $\rho_i\neq 1,\ i=1,\ldots, m$). Then $\rho$ is called \f{combinatorially free} if the following conditions are satisfied:
\begin{itemize}
 \item[\rm (i)] For any $i,j\in\{1,\ldots, m\}, i\neq j$, $\rho_i$ is not a submonomial of $\rho_j$.
 \item[\rm (ii)] For any $i,j\in \{1,\ldots, m\}$ (where not necessarily $i\neq j$) and any factorization $\rho_i=x_iy_i,\ \rho_j=x_j y_j$ into monomials $x_i,y_i,x_j,y_j\neq 1$, we have
\[
 x_i\neq y_j.
\]
In other words, the beginning of $\rho_i$ is not the ending of $\rho_j$ for any $i,j$.
\end{itemize}                                                                                                                                                                                                             
\end{defi}

Note that this definition is independent of the choice of the grading $\tau=(\tau_1,\ldots, \tau_d)$. 

\begin{theo}
\label{combfreetheo} 
Let $\rho=\{\rho_1,\ldots, \rho_m\}\in I_A$ be a sequence of monomials. Then $\rho$ is strongly free if and only if it is combinatorially free.
\end{theo}

For a proof we refer to \cite{DA}, Th.3.1.
\vspace{10pt}

The above fact can be used to derive a class of sufficient criteria for the strong freeness of arbitrary sequences of polynomials by considering only the leading monomials with respect to certain orders. By $\mf{M}$ we denote the set of all monomials (including $1$) in $A$. 

\begin{defi}
\label{multorder}
A total order $<$ on $\mf{M}$ is said to be \f{multiplicative} if the following holds:
\begin{itemize}
 \item[\rm (i)] $1 < \alpha$ for all $1\neq \alpha\in\mf{M}$,
 \item[\rm (ii)] if $\alpha<\alpha'$, then $\beta \alpha \gamma < \beta \alpha' \gamma$ for all $\beta,\gamma\in \mf{M}$.
\end{itemize}
\end{defi}

\begin{exam}
\label{lexicographicorderexam}
 Let $<$ be a total ordering on the set $X=\{X_1,\ldots, X_d\}$. Then the lexicographic ordering $<'$ on $\mf{M}$ given by
\begin{eqnarray*}
 && X_{i_1}\cdots X_{i_l} <' X_{j_1}\cdots X_{j_k}\\ 
&:\Longleftrightarrow & l<k\ \mbox{or}\ l=k\ \mbox{and}\ X_{i_n}=X_{j_n}\ \mbox{for}\ 1\le n< n_0\ \mbox{and}\ X_{i_{n_0}} < X_{j_{n_0}}
\end{eqnarray*}
is multiplicative.
\end{exam}

Let $<$ be a multiplicative order on $\mf{M}$ and 
\[
\rho=\sum_{\alpha\in \mf{M}} c_\alpha \alpha\in A,\ c_\alpha\in\F_p
\]
be an arbitrary polynomial. The \f{high term} of $f$ with respect to $<$ is the unique monomial $\alpha\in\mf{M}$ such that $c_\alpha\neq 0$ and $\alpha'<\alpha$ for all $\alpha'\neq \alpha$ satisfying $c_{\alpha'}\neq 0$. We have the following

\begin{theo}[Anick's criterion]
\label{anickscrit}
Let $\rho_1,\ldots, \rho_m\in I_A$ be a sequence of homogeneous polynomials. Let $<$ be a multiplicative order on $\mf{M}$ and let $\tilde{\rho}_i$ denote the high term of $\rho_i,\ i=1,\ldots, m$ with respect to $<$ . If the sequence $\tilde{\rho}_1,\ldots, \tilde{\rho}_m$ is combinatorially free, then $\rho_1, \ldots, \rho_m$ is strongly free.
\end{theo}
\begin{proof}
 This is shown in \cite{DA}, Th.3.2 in the case of lexicographic orders $<$ defined as in \ref{lexicographicorderexam}. As has been remarked by P.\ Forré in \cite{PF2}, Th.2.6, the proof immediately carries over to arbitrary multiplicative orders. 
\end{proof}

The above theorem yields a strong method to proof that a given sequence of polynomials is strongly free. E.g.\ the initial forms of the polynomials $\rho_1, \ldots, \rho_d$, $d\ge 4$ given in \ref{strfreeexample}(i) with respect to the lexicographic order induced by \linebreak $X_1 < X_3< \ldots < X_{d-1} < X_2 < X_4 < \ldots < X_d$ (where for simplicity we have assumed that $d$ is even) are given by
\[
 \tilde{\rho}_1= X_2 X_1,\ \tilde{\rho}_2 = X_2 X_3,\ \ldots,\ \tilde{\rho}_{d-1} = X_d X_{d-1},\ \tilde{\rho}_d = X_d X_1,
\]
which is a combinatorially free sequence.

For polynomials lying in the (free) Lie subalgebra of $\polyx$ generated by $X_1,\ldots, X_d$, further criteria have been obtained by J.\ Labute using the elimination theorem for quotients of free Lie algebras, cf.\ \cite{JL}, Th.3.3. In the proof of \ref{cohomologicalcrit}, we will make use of Anick's criterion with respect to some special multiplicative order. In order to do so, the following definition is crucial:

\begin{defi}
\label{specialorder}
 Let $U\subseteq X=\{X_1,\ldots, X_d\}$ be a subset and $<$ a total ordering on $X=\{X_1,\ldots, X_d\}$. We define a total order $<_U$ on $\mf{M}$ as follows: For a monomial $\alpha=X_{i_1} \cdots X_{i_{n_\alpha}}$ let $l^U(\alpha)$ denote the number of $X_i$'s in $\alpha$ such that $X_i\not\in U$, i.e.\  $l^U(\alpha):=\#\{k=1,\ldots, n_\alpha\ |\ X_{i_k}\not\in U\}$. If $\alpha'=X_{j_1}\cdots X_{j_{n_{\alpha'}}}$ is another monomial, we set $\alpha <_U \alpha'$ if and only if
\begin{itemize}
 \item[\rm (i)] $\deg^\tau \alpha < \deg^\tau \alpha'$ or
 \item[\rm (ii)] $\deg^\tau \alpha = \deg^\tau \alpha'$ and $l^U(\alpha) < l^U(\alpha')$ or
 \item[\rm (iii)] $\deg^\tau \alpha = \deg^\tau \alpha'$ and $l^U(\alpha) = l^U(\alpha')$ and
\[
 k^U(\alpha):=\sum_{\stackrels{1\le k \le n_\alpha,}{X_{i_k}\not\in U}} \deg^\tau (X_{i_1}\cdots X_{i_k}) < k^U(\alpha'):=\sum_{\stackrels{1\le k \le n_{\alpha'},}{X_{j_k}\not\in U}} \deg^\tau (X_{j_1}\cdots X_{j_k})
\]
or
 \item[\rm (iv)] $\deg^\tau \alpha = \deg^\tau \alpha'$ and $l^U(\alpha) = l^U(\alpha')$ and $k^U(\alpha) = k^U(\alpha')$ and $\alpha <' \alpha'$ with respect to the lexicographic ordering $<'$ induced by $<$.
\end{itemize}
\end{defi}

Roughly speaking, the more variables of $U$ are contained in a monomial $\alpha\in\mf{M}$ and the more to the right hand side these variables occur, the smaller is $\alpha$ with respect to $<_U$. 

\pagebreak
\begin{lemm}
\label{specialorderproof}
 For any total ordering $<$ on $X=\{X_1,\ldots, X_d\}$ and any subset $U\subseteq X$, the order $<_U$ as defined in \ref{specialorder} is a multiplicative order on $\mf{M}$.
\end{lemm}
\begin{proof}
 Clearly $1<_U\alpha$ for any $\alpha\in\mf{M}$. Now assume, that $\alpha, \alpha', \beta, \gamma\in\mf{M}$ where $\alpha <_U \alpha'$. If $\alpha, \alpha'$ satisfy condition (i) or (ii) of \ref{specialorder} respectively, then obviously the same holds for $\beta \alpha \gamma, \beta \alpha' \gamma$. Hence, assume that $\alpha, \alpha'$ satisfy condition (iii), i.e.\ $\deg^\tau \alpha = \deg^\tau \alpha',\ l^U(\alpha) = l^U(\alpha')$ and $k^U(\alpha) < k^U(\alpha')$. It follows that $\deg^\tau (\beta \alpha \gamma) = \deg^\tau (\beta \alpha' \gamma)$ and $l^U(\beta \alpha \gamma) = l^U(\beta \alpha' \gamma)$. Noting that for two monomials $\delta,\varepsilon\in\mf{M}$
\[
 k^U(\delta \varepsilon) = k^U(\delta) + l^U(\varepsilon) \deg^\tau \delta + k^U(\varepsilon)
\]
we conclude that
\begin{eqnarray*}
 k^U(\beta \alpha \gamma) & = & k^U(\beta) + l^U(\alpha \gamma)\deg^\tau \beta + k^U(\alpha \gamma) \\
& = & k^U(\beta) + l^U(\alpha \gamma)\deg^\tau(\beta) + k^U(\alpha) + l^U(\gamma)\deg^\tau(\alpha) + k^U(\gamma) \\
& < & k^U(\beta) + l^U(\alpha' \gamma)\deg^\tau(\beta) + k^U(\alpha') + l^U(\gamma)\deg^\tau(\alpha') + k^U(\gamma) \\
& = & k^U(\beta \alpha' \gamma)
\end{eqnarray*}
and therefore $\beta \alpha \gamma <_U \beta \alpha' \gamma$ by condition (iii). Finally, the same argument shows that if $\alpha, \alpha'$ satisfy condition (iv), then this also holds for $\beta \alpha \gamma, \beta \alpha' \gamma$.
\end{proof}

\section{Higher Massey products and cohomological criteria for mildness}

By definition a finitely presented pro-$p$-group is mild if it possesses a strongly free presentation in terms of generators and relations. Since in general a mild pro-$p$-group also admits presentations which are \it not \rm strongly free, it is desirable to have sufficient criteria for mildness that do not depend on the choice of a specific minimal generating system for a group. As mentioned in the introduction, such a criterion is given by the \it cup-product criterion \rm proven by A.\ Schmidt (cf.\ \cite{AS2}, Th.6.2) using Labute's results on strongly free sequences in free Lie algebras. As a necessary condition, the cup-product $H^1(G,\F_p)\otimes H^1(G,\F_p)\stackrel{\cup}{\longrightarrow} H^2(G,\F_p)$ needs to be surjective. The main goal of this section is to prove Theorem \ref{cohomologicalcrit}, which gives a generalization to arbitrary higher \f{Massey products}. In particular, the cup-product can be trivial.

Massey products have been introduced by W.S.\ Massey as higher analogs of the cup-product in algebraic topology. They can be studied in a context as general as the cohomology of complexes, i.e.\ they can be defined for various cohomology theories (e.g.\ see \cite{DK}, \cite{JPM} and \cite{CD}). For the applications we have in mind, we introduce them for group cohomology of pro-$p$-groups with $\F_p$-coefficients.

Let $G$ be a pro-$p$-group. Recall that me make the notational convention $H^i(G)=H^i(G,\F_p)$. We denote by $\CC^*(G)$ the standard inhomogeneous cochain complex (e.g.\ see \cite{NSW}, Ch.I, §2) for the trivial $G$-module $\F_p$. 

\begin{defi}
 Let $n\ge 2$ and $\alpha_1,\ldots, \alpha_n\in H^1(G)$. We say that the \f{$n$-th Massey product $\langle \alpha_1,\ldots, \alpha_n\rangle_n$} is \f{defined} if there is a collection
\[
 \mc{A}=\{a_{ij}\in \CC^1(G)\ |\ 1\le i,j\le n,\ (i,j)\neq (1,n)\}
\]
(called a \f{defining system} for $\langle \alpha_1,\ldots, \alpha_n\rangle_n$), such that the following conditions hold:
\begin{itemize}
 \item[\rm (i)] $a_{ii}$ is a representative of the cohomology class $\alpha_i,\ 1\le i\le n$.
 \item[\rm (ii)] For $1\le i < j\le n,\ (i,j)\neq (1,n)$ it holds that
 \[
  \partial^2 (a_{ij}) = \sum_{l=i}^{j-1} a_{il}\cup a_{(l+1)j}
 \]
\end{itemize}
where $\partial^2$ denotes the coboundary operator $\partial^2: \CC^1(G)\longrightarrow \CC^2(G)$.
If $\mc{A}$ is a defining system for $\langle \alpha_1,\ldots, \alpha_n\rangle_n$, the $2$-cochain
\[
 b_\mc{A} = \sum_{l=1}^{n-1} a_{1l}\cup a_{(l+1)n}
\]
is a cocyle and we denote its class in $H^2(G)$ by $\langle \alpha_1,\ldots, \alpha_n\rangle_\mc{A}$. We set
\[
 \langle \alpha_1,\ldots, \alpha_n\rangle_n = \bigcup_{\mc{A}} \langle \alpha_1,\ldots, \alpha_n\rangle_\mc{A}
\]
where $\mc{A}$ runs over all defining systems. The Massey product $\langle \alpha_1,\ldots, \alpha_n\rangle_n$ is called \f{uniquely defined} if $\#\langle \alpha_1,\ldots, \alpha_n\rangle_n=1$. The $n$-th Massey product is \f{uniquely defined for $G$} if $\langle \alpha_1,\ldots, \alpha_n\rangle_n$ is uniquely defined for all $\alpha_1,\ldots, \alpha_n\in H^1(G)$.
\end{defi}

It can be shown that the $n$-th Massey product is uniquely defined if the Massey products of lower order are uniquely defined and identically zero. More precisely, we have the following

\begin{prop}
 Let $n\ge 2$ and $\alpha_1,\ldots, \alpha_n\in H^1(G)$.
\begin{itemize}
  \item[\rm (i)] For $n=2$ the Massey product $\langle \alpha_1,\alpha_2\rangle_2$ is uniquely defined and given by the cup-product, i.e.\
\[
 \langle \alpha_1,\alpha_2\rangle_2 = \alpha_1 \cup \alpha_2.
\]
\item[\rm (ii)] Let $n\ge 2$. Assume that for all $2\le l<n$ and all $\alpha_1,\ldots, \alpha_l$ the $l$-th Massey product $\langle \alpha_1,\ldots, \alpha_l\rangle_l=0$ is uniquely defined and given by the zero class in $H^2(G)$. Then for all $\beta_1,\ldots, \beta_n\in H^1(G)$ the $n$-th Massey product $\langle \beta_1,\ldots, \beta_n \rangle_n$ is also uniquely defined and yields a multilinear map (of $\F_p$-vector spaces)
\[
\langle \cdot, \ldots, \cdot \rangle_n: H^1(G)^n\longrightarrow H^2(G).
\]
\end{itemize}
\end{prop}
\begin{proof}
 Statement (i) follows immediately from the definition. For a proof of (ii) see \cite{DK}, Lemma 20 and \cite{RF}, Lemma 6.2.4.
\end{proof}

\begin{rema}
\label{masseycompatibility}
Higher Massey products satisfy the same functoriality properties as the cup-product. In particular, provided they are uniquely defined,  they commute with inflation, restriction and corestriction.
\end{rema}

Now assume that the pro-$p$-group $G$ is finitely presented with $d=h^1(G)$. Let 
\[
\begin{tikzpicture}[description/.style={fill=white,inner sep=2pt}, bij/.style={below,sloped,inner sep=1.5pt}]
\matrix (m) [matrix of math nodes, row sep=1.5em,
column sep=2.5em, text height=1.5ex, text depth=0.25ex]
{  1 & R & F & G & 1\\};
\path[->,font=\scriptsize]
(m-1-1) edge node[auto] {} (m-1-2)
(m-1-2) edge node[auto] {} (m-1-3)
(m-1-3) edge node[auto] {} (m-1-4)
(m-1-4) edge node[auto] {} (m-1-5);
\end{tikzpicture}
\]
be a minimal presentation of $G$ where $F$ is a free pro-$p$-group on generators $x_1, \ldots, x_d$. Let $\chi_1, \ldots, \chi_d\in H^1(F)=H^1(G)$ be the dual basis corresponding to $x_1, \ldots, x_d$. The five term exact sequence yields the isomorphism
\[
 \begin{tikzpicture}[description/.style={fill=white,inner sep=2pt}, bij/.style={above,sloped,inner sep=1.5pt}, column 1/.style={anchor=base east},column 2/.style={anchor=base west}]
 \matrix (m) [matrix of math nodes, row sep=1.5em,
 column sep=2.5em, text height=1.5ex, text depth=0.25ex]
 {\tg:\ H^1(R)^G  & H^2(G).\\
};
 \path[->,font=\scriptsize]
 (m-1-1) edge node[bij] {$\sim$} (m-1-2);
 \end{tikzpicture}
\]
Hence, every element $r \in R$ gives rise to the \f{trace map}
\[
 \begin{tikzpicture}[description/.style={fill=white,inner sep=2pt}, bij/.style={above,sloped,inner sep=1.5pt}, column 1/.style={anchor=base east},column 2/.style={anchor=base west}]
 \matrix (m) [matrix of math nodes, row sep=1.5em,
 column sep=2.5em, text height=1.5ex, text depth=0.25ex]
 {\tr_r:\ H^2(G) & \mathbb{F}_p,\\
  \varphi& (\tg^{-1}\varphi)(r).\\
};
 \path[->,font=\scriptsize]
 (m-1-1) edge node[auto] {} (m-1-2);
 \path[|->,font=\scriptsize]
 (m-2-1) edge node[auto] {} (m-2-2);
 \end{tikzpicture}
\]
If $r_1,\ldots, r_m$ is a minimal system of defining relations, i.e.\ a minimal system of generators of $R$ as closed normal subgroup of $F$, then $\tr_{r_1},\ldots, \tr_{r_m}$ is a basis of $H^2(G)^\vee$.

Recall that we have the topological isomorphism

\[
 \begin{tikzpicture}[description/.style={fill=white,inner sep=2pt}, bij/.style={above,sloped,inner sep=1.5pt}, column 1/.style={anchor=base east},column 2/.style={anchor=base west}]
 \matrix (m) [matrix of math nodes, row sep=1.5em,
 column sep=2.5em, text height=1.5ex, text depth=0.25ex]
 { \Omega_F  & \power,\ x_i &1+X_i.\\
};
 \path[->,font=\scriptsize]
 (m-1-1) edge node[bij] {$\sim$} (m-1-2);
 \path[|->,font=\scriptsize]
 (m-1-2) edge node[auto] {} (m-1-3);
 \end{tikzpicture}
\]
By $\psi: F\hookrightarrow \power$ we denote the composite  of the map $F\hookrightarrow\Omega_F, f\mapsto f-1$ with the above isomorphism, mapping $F$ into the augmentation ideal of $\power$. The element $\psi(f)$ is called \f{Magnus expansion} of $f\in F$. 

As noticed by H.\ Koch, there is a close connection between higher Massey products for $G$ and its defining relations lying in higher Zassenhaus filtration steps. In order to make this precise, we need the following notation: 

\begin{defi}\quad
\label{magnusdefi}
\begin{itemize}
 \item[\rm (i)] A \f{multi-index $I$ of height $d$ and length $|I|=k$} is a tuple of elements $I=(i_1, \ldots, i_k)\in\N^k$ where $k$ is a natural number and $1\le i_j\le d$ for $1\le j\le k$. We denote by $\mc{M}_d^k$ the set of all multi-indices of height $d$ and length $k$.
 \item[\rm (ii)] For any multi-index $I$ let $\varepsilon_{I,p}(f)\in\F_p$ denote the coefficient corresponding to $I$ in the Magnus expansion of $f$, i.e.
\[
 \psi(f)=\sum_I \varepsilon_{I,p}(f) X_I
\]
where $I$ runs over all multi-indices of height $d$ and $X_I$ denotes the monomial $X_I=X_{i_1}\cdots X_{i_l}$ for any $I=(i_1, \ldots, i_l)$.
\end{itemize}
\end{defi}

By definition we have $f\in F_{(n)},\ n\ge 1$ if and only if $\varepsilon_{I,p}(f)=0$ for all multi-indices $I$ of length $|I|<n$. We can now state the important

\begin{theo}
\label{epsilonmaps}
Let $G$ be a finitely presented pro-$p$-group and
\[
\begin{tikzpicture}[description/.style={fill=white,inner sep=2pt}, bij/.style={below,sloped,inner sep=1.5pt}]
\matrix (m) [matrix of math nodes, row sep=1.5em,
column sep=2.5em, text height=1.5ex, text depth=0.25ex]
{  1 & R & F & G & 1\\};
\path[->,font=\scriptsize]
(m-1-1) edge node[auto] {} (m-1-2)
(m-1-2) edge node[auto] {} (m-1-3)
(m-1-3) edge node[auto] {} (m-1-4)
(m-1-4) edge node[auto] {} (m-1-5);
\end{tikzpicture}
\]
be a minimal presentation. Assume that $R\subseteq F_{(n)}$ for some $n\ge 2$. Then for all $k\le n$ the $k$-fold Massey product
\[
 \langle \cdot, \ldots, \cdot\rangle_k: H^1(G)^k\longrightarrow H^2(G),\ 2\le k\le n
\]
is uniquely defined. Furthermore, for all multi-indices $I$ of height $d$ and length $1<|I|\le n$ we have the equality
\[
 \varepsilon_{I,p}(r) = (-1)^{|I|-1} \tr_r \langle \chi_I\rangle_{|I|}
\]
for all $r\in R$ where for $I=(i_1, \ldots, i_k)$ we have set $\chi_I = (\chi_{i_1},\ldots, \chi_{i_k})\in H^1(G)^k$.  In particular, for $1 < k < n$ the $k$-fold Massey product on $H^1(G)$ is identically zero.
\end{theo}

For a proof we refer to \cite{DV2}, Prop.1.2.6 or \cite{MM2}, Th.2.2.2 respectively. By the above theorem, if $R\subseteq F_{(n)}$, the $n$-fold Massey product of $G$ is uniquely defined and determined by the Magnus expansions of a minimal system of defining relations. 

We introduce the following invariant of $G$:

\begin{defi}
\label{zassinvdefi}
 Let $G$ be a finitely generated pro-$p$-group. We define the \f{Zassenhaus invariant} $\z(G)\in\N\cup\{\infty\}$ to be the supremum of all natural numbers $n$ satisfying one (and hence all) of the following equivalent conditions:
\begin{itemize}
 \item[\rm (i)] If $1\to R\to F\to G\to 1$ is a minimal presentation of $G$, then $R\subseteq F_{(n)}$.
 \item[\rm (ii)] If $1\to R\to F\to G\to 1$ is a minimal presentation of $G$, then the induced homomorphism of graded restricted Lie algebras $\xyalign\xymatrix{\gr F\ar@{->>}[r] & \gr G}$ is injective in degrees $< n$.
 \item[\rm (iii)] The $k$-fold Massey product $H^1(G)^k\to H^2(G)$ is uniquely defined and identically zero for $2\le k< n$.
\end{itemize}
\end{defi}

The equivalence (i)$\Leftrightarrow$(ii) is clear from the definition and (i)$\Leftrightarrow$(iii) is a direct consequence of \ref{epsilonmaps}. Obviously, by definition we have $\z(G)\ge 2$. Furthermore, $\z(G)=\infty$ if and only if $G$ is free.

The coefficients $\varepsilon_{I,p}(f)$ satisfy certain symmetry relations which by \ref{epsilonmaps} imply analogous relations for higher Massey products. In order to state them, we need the notion of \f{shuffles}:

\begin{defi}
 Let $a,b\ge 1$ be integers. A \f{$(a,b)$-shuffle} $f$ is a bijection $\xyalign\xymatrix{f: \{1,\ldots, a+b\} \ar[r]^-\sim & \{1,\ldots, a+b\}}$ such that $f(i)<f(j)$ if $i<j\le a$ or $a+1\le i <j$.
\end{defi}

The following \f{shuffle identities} have been conjectured by W.\ G.\ Dwyer in a topological setting (cf.\ \cite{WGD}, Conj.4.6) and proven in the cohomology of finitely presented pro-$p$-groups by D.\ Vogel (cf.\ \cite{DV2}, Cor.1.2.10):

\begin{prop}[$(a,b)$-shuffle identity]
\label{shuffleproperties}
Let $G$ be a finitely generated pro-$p$-group with Zassenhaus invariant $\mf{z}(G)\ge n$. Let $a,b\ge 1$ be integers satisfying $a+b=n$ and $\xi_1,\ldots, \xi_n\in H^1(G)$. Then
\[
 \sum_f \langle \xi_{f^{-1}(1)}, \ldots, \xi_{f^{-1}(n)} \rangle_n = 0
\]
where the sum is taken over all $(a,b)$-shuffles $f$.
\end{prop}

Note that for $n=2$ this simply amounts to the anti-commutativity of the cup-product. If $\z(G)=3$, the above proposition is equivalent to
\begin{eqnarray*}
\label{shuffle3}
 \langle \xi_1, \xi_2, \xi_3\rangle_3 + \langle \xi_2, \xi_3, \xi_1\rangle_3 + \langle \xi_3, \xi_1, \xi_2\rangle_3 = 0,\quad \langle \xi_1, \xi_2, \xi_3\rangle_3 = \langle \xi_3, \xi_2, \xi_1\rangle_3
\end{eqnarray*}
for all $\xi_1, \xi_2, \xi_3\in H^1(G)$. In particular, for $p\neq 3$ we have $\langle \xi_1, \xi_1, \xi_1\rangle_3 = 0$.

We can now state and prove our main result:

\begin{theo}
\label{cohomologicalcrit}
 Let $p$ be a prime number and $G$ a finitely presented pro-$p$-group with $n=\z(G)<\infty$. Assume that $H^1(G)$ admits a decomposition $H^1(G)=U\oplus V$ as $\F_p$-vector space such that for some natural number $e$ with $1\le e \le n-1$ the $n$-fold Massey product $\langle\cdot, \ldots, \cdot \rangle_n: H^1(G)^n \longrightarrow H^2(G)$ satisfies the following conditions:
\begin{itemize}
 \item[\rm (a)] $\langle \xi_1,\ldots, \xi_n \rangle_n =0$ for all tuples $(\xi_1,\ldots, \xi_n)\in H^1(G)^n$ such that $\#\{i\ |\ \xi_i\in V\}\ge n-e+1$.
 \item[\rm (b)] $\langle\cdot, \ldots, \cdot \rangle_n$ maps 
\[
U^{\otimes e} \otimes V^{\otimes n-e}
\]
surjectively onto $H^2(G)$.
\end{itemize}
Then $G$ is mild (with respect to the Zassenhaus filtration). In particular, $G$ is of cohomological dimension $\cd G=2$.
\end{theo}
\begin{proof}
We set $d=h^1(G),\ m=h^2(G)$ and $c=\dimfp U$. Furthermore, we choose bases $\chi_1,\ldots, \chi_c$ and $\chi_{c+1},\ldots, \chi_d$ of $U$ and $V$ respectively. Let $\overline{x}_1, \ldots, \overline{x}_d\in G$ be arbitrary lifts of the basis of $H_1(G)=G/G^p[G,G]$ dual to $\{\chi_1, \ldots, \chi_d\}$. Then $G$ admits a minimal presentation
\[
 \begin{tikzpicture}[description/.style={fill=white,inner sep=2pt}, bij/.style={above,sloped,inner sep=1.5pt}, column 1/.style={anchor=base east},column 2/.style={anchor=base west}]
 \matrix (m) [matrix of math nodes, row sep=1.5em,
 column sep=2.5em, text height=1.5ex, text depth=0.25ex]
 {1 & R & F & G & 1\\
};
 \path[->,font=\scriptsize]
 (m-1-1) edge node[auto] {} (m-1-2)
 (m-1-2) edge node[auto] {} (m-1-3)
 (m-1-3) edge node[auto] {$\pi$} (m-1-4)
 (m-1-4) edge node[auto] {} (m-1-5);
 \end{tikzpicture}
\]
where $F$ is the free pro-$p$-group on generators $x_1, \ldots, x_d$ and $\pi$ maps $x_i$ to $\overline{x}_i$. Recall that we have the embedding $\gr F\hookrightarrow \polyx$ (of graded restricted Lie algebras) where $X=\{X_1,\ldots, X_d\}$, mapping the initial form of $x_i$ to $X_i$. Let $<$ denote the natural order on $X$, i.e.\ $X_1 < X_2 < \ldots < X_d$. We order the set of monomials by the order $<_U$ introduced in \ref{specialorder} where by abuse of notation we denote the subset $\{X_1,\ldots, X_c\}\subset X$ also by $U$. We have seen in \ref{specialorderproof} that $<_U$ is in fact a multiplicative order. Consider the following subset of the set $\mc{M}_d^n$ of multi-indices of height $d$ and length $n$:
\[
 B=\{(i_1,\ldots, i_n)\in \mc{M}_d^n |\ i_1,\ldots, i_e\le c\ \mbox{and}\ i_{e+1},\ldots, i_n>c\}.
\]
Note that $b:=\# B = \dimfp  (U^{\otimes e} \otimes V^{\otimes n-e})=c^e (d-c)^{n-e}$. Since by condition $(b)$ the homomorphism $\varphi: U^{\otimes e} \otimes V^{\otimes n-e}\longrightarrow H^2(G)$ is surjective, there exists a basis $C=\{y_1,\ldots, y_m\}$ of $H^2(G)$, such that the transformation matrix $M$ of $\varphi$ with respect to the bases
\[
 \mc{B}=\{\chi_{i_1}\otimes \cdots \otimes \chi_{i_n}\ |\ (i_1,\ldots, i_n)\in B\}
\]
(which we order via $<_U$) and $C$ is of the form
\begin{eqnarray*}
\label{echelonform}
M= \begin{pmatrix} 
0 & \cdots & 0 & 1 & * & \cdots & * & * & * & \cdots & * & \cdots & * & * & \cdots & *\\ 
0 & \cdots & 0 & 0 & 0 & \cdots & 0 & 1 & * & \cdots & * & \cdots & * & * & \cdots & *\\
\vdots & \vdots  & \vdots  & \vdots  & \vdots  & \vdots  & \vdots  & \vdots  & \vdots  & \vdots  & \vdots & \ddots & \vdots & \vdots  & \vdots  & \vdots \\
0 & \cdots & 0 & 0 & 0 & \cdots & 0 & 0 & 0 & \cdots & 0 & \cdots & 1 & * & \cdots & *\\
\end{pmatrix}.
\end{eqnarray*}
In fact, first choose an arbitrary basis of $H^2(G)$ and transform the corresponding matrix $M'$ of $\varphi$ into row echelon form by applying elementary row operations, noting that $\operatorname{rank} M' = m$. For $1\le j\le m$ and $I\in B$ we denote by $m_{j,I}$ the coefficient in the $j$-th row and the column corresponding to $I$ of $M$. We choose $r_1, \ldots, r_m\in R$ such that for $1\le j\le m$ the image $\ol{r}_j$ of $r_j$ in $R/R^p[F,R]$ is dual to $\tg^{-1}(y_i)\in H^1(R)^G=(R/R^p[F,R])^\vee$ where as above $\tg$ denotes the transgression isomorphism
\[
 \begin{tikzpicture}[description/.style={fill=white,inner sep=2pt}, bij/.style={above,sloped,inner sep=1.5pt}, column 1/.style={anchor=base east},column 2/.style={anchor=base west}]
 \matrix (m) [matrix of math nodes, row sep=1.5em,
 column sep=2.5em, text height=1.5ex, text depth=0.25ex]
 {\tg:\ H^1(R)^G  & H^2(G).\\
};
 \path[->,font=\scriptsize]
 (m-1-1) edge node[bij] {$\sim$} (m-1-2);
 \end{tikzpicture}
\]
Hence, $r_1,\ldots, r_m$ is a minimal system of defining relations of $G$ and the trace map $\tr_{r_j}\in H^2(G)^\vee$ is the linear form dual to $y_j$. Let $\rho_j\in \gr F$ be the initial form of $r_j$. We claim that the sequence $\rho_1, \ldots, \rho_m$ is strongly free. Since $\z(G)=n$, by \ref{epsilonmaps} and the definition of the matrix $M=(m_{j,I})$, we have
\[
\varepsilon_{I,p}(r_j) = (-1)^{|I|-1} \tr_{r_j} \langle \chi_I\rangle_n = (-1)^{|I|-1} m_{j,I},\ \mbox{for all}\ 1\le j\le m,\ I\in B.
\]
First note that since every row of $M$ is non-zero, this implies $\rho_j\in gr_n F=F_{(n)}/F_{(n+1)}$, i.e.\ the $\rho_j$ are homogeneous polynomials of degree $n$.
Noting in addition that the initial form $\rho_j$ is obtained by leaving out all monomials of degree $> n$ in the Magnus expansion of $r_j$, condition (a) gives rise to
\[
 \varepsilon_{I,p}(r_j) = (-1)^{|I|-1} \tr_{r_j} \langle \chi_I\rangle_n = 0
\]
for $1\le j\le m$ and any multi-index $I=(i_1,\ldots, i_n)$ satisfying the condition $\#\{k\ |\ i_k > c\}\ge n-e+1$, which is equivalent to $\#\{k\ |\ i_k \le c\}\le e-1$. That is, every monomial of $\rho_j$ contains at least $e$ factors $X_i,\ i\le c$. Since the monomials in $X_I, I\in B$ have the property that the $X_i, i> c$ are on the right end, by definition of the ordering $<_U$ (cf.\ \ref{specialorder}) we conclude that the high terms $\tilde{\rho}_j$ of the $\rho_j$ with respect to $<_U$ are monomials of the form $\tilde{\rho}_j = X_{I_j}$ for some $I_j$ in $B$. Furthermore, taking into account that the matrix $M$ is in row echelon form, we see that the $\tilde{\rho}_j$ are pairwise distinct. In particular, $\tilde{\rho}_j$ is not contained in $\tilde{\rho}_k$ for $1\le j,k\le m,\ j\neq k$ and since each $\tilde{\rho}_j$ begins (from the left) with $e$ variables $X_i, i\le c$ and ends with $n-e$ variables $X_i, i > c$, the beginning of $\tilde{\rho}_j$ is never the ending of $\tilde{\rho}_k$ for $1\le j,k\le m$. In other words, the sequence $\tilde{\rho}_1,\ldots, \tilde{\rho}_m$ is combinatorially free and by Anick's criterion \ref{anickscrit} we conclude that $\rho_1, \ldots, \rho_m$ is strongly free.
\end{proof}

\begin{rema}\quad
\label{cohomologicalcritexamples}
\begin{itemize}
 \item[\rm (i)] If $\z(G)=2$, the above statement is the following: Assume that \linebreak $H^1(G)=U\oplus V$ and the cup-product $H^1(G)\otimes H^1(G)\stackrel{\cup}{\to} H^2(G)$ is trivial on $V\otimes V$ and maps $U\otimes V$ surjectively onto $H^2(G)$. Then $G$ is mild. Hence, we have obtained new proof for the cup-product criterion including the case $p=2$.
 \item[\rm (ii)] Next consider the case $\z(G)=3$ and apply \ref{cohomologicalcrit} for $e=1$. We obtain that $G$ is mild if $H^1(G)=U\oplus V$ and the triple Massey product 
is trivial on $V\otimes V\otimes V$ and maps $U\otimes V\otimes V$ surjectively onto $H^2(G)$. A straightforward application of \cite{DV2}, Prop.1.3.3 shows that these conditions are equivalent to the following: $G$ possesses a minimal presentation $G=\langle x_1,\ldots, x_d\ |\ r_1,\ldots, r_m\rangle$ where  
\[ 
r_n\equiv\left\{ \begin{array}{ll} 
\displaystyle\prod_{\stackrels{1\le i < j\le d,}{1\le k\le j}} [[x_i, x_j], x_k]^{a_{ijk}^n} \mod F_{(4)},& \mbox{if}\ p\neq 3,\\
\displaystyle\prod_{1\le i\le d} x_i^{3a_i^n} \cdot \displaystyle\prod_{\stackrels{1\le i < j\le d,}{1\le k\le j}} [[x_i, x_j], x_k]^{a_{ijk}^n} \mod F_{(4)},& \mbox{if}\ p= 3\\
\end{array}\right.
\]
for $n=1, \ldots, m$ with $a_i^n, a_{ijk}^n\in\mathbb{F}_p$, such that for some $1\le c <d$:
\begin{itemize}
\item[\rm (I)] $a_{ijk}^n=0$ if  $c<i<j, c<k\le j$ and $1\le n\le m$,
 \item[\rm (II)] The $m\times c(d-c)^2$-matrix
\[
 (a_{ijk}^n)_{n,(ijk)},\quad 1\le n\le m,\ 1\le i\le c < k\le j\ \mbox{or}\ 1\le k\le c < i < j 
\]
has rank $m$.
 \item[\rm (III)] If $p=3$, then $a_i^n=0$ if $c<i$ and $1\le n\le m$.
\end{itemize}
These conditions resemble \cite{AS3}, Th.5.5 and \cite{PF2}, Cor.6.5 respectively, where similar statements for relations of degree $2$ are given.

Applying \ref{cohomologicalcrit} for $e=2$, we find that $G$ is mild if $H^1(G)=U\oplus V$, the triple Massey product is trivial on the subspaces 
\[
H^1(G)\otimes V\otimes V,\quad V\otimes H^1(G)\otimes V,\quad V\otimes V\otimes H^1(G)
\]
 and maps $U\otimes U\otimes V$ surjectively onto $H^2(G)$. As in the case $e=1$ these conditions translate into basic commutators.
\end{itemize}
\end{rema}

We end this section by giving some remarks on arithmetic examples: In \cite{JL}, J.\ Labute applied the theory of mild pro-$p$-groups to give the first examples of Galois groups of the form $G_S(p)=\Gal(\Q_S(p)|\Q)$ satisfying $\cd G_S(p)= 2$ where $\Q_S(p)$ is the maximal pro-$p$-extension $\Q$ unramified outside a finite set $S$ of prime numbers different from $p$. A.\ Schmidt proved that for a general number field the sets $S$ such that $\cd G_S(p)= 2$ are cofinal among all finite sets $S$ even in the \f{tame case}, i.e.\ considering primes with residue characteristic different from $p$ only (the precise statement is stronger, see \cite{AS2}, Th.1.1 for details). An important ingredient in the proof is the cup-product criterion for mildness. The question naturally arises whether mild pro-$p$-groups with Zassenhaus invariant $\ge 3$ occur as arithmetical Galois groups of the form $G_S(p)$. For $p=2$ a positive answer is given by the following two examples studied in \cite{JGArith}, which actually satisfy the conditions given in \ref{cohomologicalcrit}:

\begin{exam}\quad
\label{arithmeticexamples}
\begin{itemize}
\item[(i)] Let $S=\{l_0,l_1,\ldots, l_n\}$ for some $n\ge 1$ and prime numbers $l_0=2,\ l_i\equiv 9\mod 16,\ i=1,\ldots, n$, such that the Legendre symbols satisfy
\[
 \left(\frac{l_i}{l_j}\right)_2=1,\ 1\le i,j\le n,\ i\neq j.
\]
Then $G_S(2)$ is a mild pro-$2$-group with generator rank $h^1(G_S(2))=$\linebreak $n+1$, relation rank $h^2(G_S(2))=n$ and Zassenhaus invariant $\z(G_S(2))=3$.
\item[(ii)]  Let $S=\{2, 17, 7489, 15809\}, T=\{5\}$. Then the Galois group $G_S^T(2)$ of the maximal pro-$2$-extension of $\Q$ unramified outside $S$ and completely decomposed at $T$ is a mild pro-$2$-group with $h^1(G_S^T(2))=h^2(G_S^T(2))=\z(G_S^T(2))=3$.
\end{itemize}
\end{exam}

Moreover, assuming the Leopoldt conjecture for $p=2$, the latter group $G_S^T(2)$ is a \f{fab pro-$2$-group}, i.e.\ the abelianization $H^{ab}=H/[H,H]$ is finite for every open subgroup $H\subseteq G_S^T(2)$. To the author's knowledge this yields the first known example of a mild fab group with Zassenhaus invariant $\ge 3$ and also the first example of a mild fab group with generator rank $\le 3$. J.\ Labute, C.\ Maire and J.\ Miná\v{c} have announced analogous results for odd $p$.

\section{One-relator pro-$p$-groups}

In the following two sections, we investigate the cohomological dimension of finitely generated pro-$p$-groups with a single defining relation. The structure of these groups is also of arithmetic interest, e.g.\ \f{Demu\v{s}kin groups} occur naturally as the Galois group of the maximal $p$-extension of a local field containing a primitive $p$-th root of unity or as the pro-$p$-completion of the (discrete) fundamental group of a compact Riemann surface. We will study pro-$p$-groups of \f{Demu\v{s}kin type} in section 6.

\begin{defi}
 A \f{one-relator pro-$p$-group} is a pro-$p$-group $G$ satisfying \linebreak $h^1(G)< \infty$ and $h^2(G)=1$, i.e.\ $G\cong F/(r)$ where $F$ is a finitely generated free pro-$p$-group and $r\in F_{(2)}$. 
\end{defi}

The goal of this section is the proof of Th.\ \ref{onerelatorgrouptheo}. A crucial ingredient is the structure of the restricted Lie algebra $\gr^\tau F$ for a finitely generated pro-$p$-group $F$.

Let $k$ be a field of characteristic $p$ and $\kpolyx$ the free associative algebra over $k$ on $X$. By $L(X), L_{res}(X)$ we denote the Lie subalgebra and the restricted Lie subalgebra generated by $X$ respectively. Then $L(X)$ is the free Lie algebra over $k$ on $X$, $L_{res}(X)$ is the free restricted Lie algebra over $k$ on $X$ and $\kpolyx$ is the universal enveloping algebra for both $L(X)$ and $L_{res}(X)$. We endow $L(X)$ and $L_{res}(X)$ with the grading induced by the natural grading on $\kpolyx$ given by $\deg X_i=\tau_i$. We recall the notion of \f{Hall commutators}:

\begin{defi}
 The set $C_n\subseteq L(X)$ of \f{Hall commutators of weight $n$} together with a total order $<$ is inductively defined as follows:
\begin{itemize}
 \item[\rm (i)] $C_1=\{X_1,\ldots, X_d\}$ with the ordering $X_1 > \ldots > X_d$.
 \item[\rm (ii)] $C_n$ is the set of all commutators $[c_1,c_2]$ where $c_1\in C_{n_1},\ c_2\in C_{n_2}$ such that $n_1+n_2=n$, $c_1>c_2$ and if $n_1\neq 1,\ c_1=[c_3,c_4]$ we have $c_2\ge c_4$. The set $C_n$ is ordered lexicographically, i.e.\ $[c_1,c_2]<[c_1',c_2']$ if and only if $c_1<c_1'$ or $c_1=c_1'$ and $c_2<c_2'$. Finally for $c\in C_n$ we set $d<c$ for any $d\in C_i,\ i<n$. 
\end{itemize}
We set $\mc{C}=\bigcup_{n\in\N} C_n$.
\end{defi}

The following lemma is well-known:
\pagebreak

\begin{lemm}
 \label{hallbasis}
Let $L(X)$ and $L_{res}(X)$ be endowed with the natural grading, i.e.\ $\tau_1=\ldots=\tau_d=1$.
\begin{itemize}
 \item[\rm (i)] The sets $C_n$ are $k$-vector space bases of $L(X)_n$. In particular, $\mc{C}=\bigcup_{n\in\N} \mc{C}_n$ is a $k$-basis of $L(X)$.
 \item[\rm (ii)] The sets
\[
\overline{C}_n= \displaystyle\dot{\bigcup}_{ip^j=n} (C_i)^{p^{j}}
\]
are $k$-bases of $L_{res}(X)_n$. In particular, $\overline{\mc{C}}=\bigcup_{n\in\N} \ol{C}_n$ is a $k$-basis of $L_{res}(X)$.
\end{itemize}
\end{lemm}

Let $\tau=(\tau_1,\ldots, \tau_d)$ be a sequence of integers $\tau_i>0$ and $\kpolyx = \bigoplus_{n\ge 0} \kpolyx^\tau_n$ be the $(X,\tau)$-grading given by $\deg_\tau X_i = \tau_i$. These gradings induce gradings $L(X)=\bigoplus_{n\ge 0} L(X)^\tau_n,\ L_{res}(X)=\bigoplus_{n\ge 0} L_{res}(X)^\tau_n$ making $L(X)$ and $L_{res}(X)$ into a graded $k$-Lie algebra and a graded restricted $k$-Lie algebra respectively, i.e.\ $[L(X)_i^\tau,L(X)_j^\tau]\subseteq L(X)_{i+j}^\tau$ and $[L_{res}(X)_i^\tau,L_{res}(X)_j^\tau]\subseteq L_{res}(X)_{i+j}^\tau$,\linebreak $(L_{res}(X)_i^\tau)^p \subseteq L_{res}(X)_{pi}^\tau$.

Note that since for any $\tau$ the set $\ol{\mc{C}}$ consists of $(X,\tau)$-homogeneous polynomials, we obtain the following

\begin{coro}
 \label{hallbasiscoro}
The set
 \[
\ol{C}^\tau_n=\{c\in\ol{\mc{C}}\ |\ c\in L_{res}(X)_n^\tau\}\subset \ol{\mc{C}}
\]
is a $k$-basis for $L_{res}(X)_n^\tau$ for all $n\in\N$.
\end{coro}

Now let $F$ be the free pro-$p$-group on $x=\{x_1,\ldots, x_d\}$ endowed with the $(x,\tau)$-grading. Recall that the map $F\hookrightarrow \Omega_F,\ f\mapsto f-1$ and the identification $\Omega_F\cong \power$ induce the embedding of graded restricted Lie algebras
\[
\begin{tikzpicture}[description/.style={fill=white,inner sep=2pt}, bij/.style={above,sloped,inner sep=1.5pt}]
\matrix (m) [matrix of math nodes, row sep=3em,
column sep=2.5em, text height=1.5ex, text depth=0.25ex]
{ \phi_F: \gr^\tau F & \polyx^\tau,\\};
\path[right hook->,font=\scriptsize]
(m-1-1) edge node[auto] {} (m-1-2);
\end{tikzpicture}
\]
mapping the initial form of $x_i$ to $X_i$.

In \cite{ML}, App.3, Th.3.5 it is shown that for $\tau=(1,\ldots, 1)$ the restricted $F_p$-Lie algebra $\gr F$ is generated by the initial forms of $x_1,\ldots, x_d$. Note that Lazard uses a different definition for the Zassenhaus filtration, however it is not hard to show that it coincides with the definition given in section 2. Furthermore, applying similar arguments to the ones in \cite{ML}, it follows more generally that $\gr^\tau F$ is generated by the initial forms of $x_1,\ldots, x_d$ as restricted $\F_p$-Lie algebra for any $\tau=(\tau_1,\ldots, \tau_d)$. This implies the following

\begin{prop}
\label{grftheo}
 Under the embedding $\phi_F$, the graded restricted Lie algebra $\gr^\tau F$ identifies with the free restricted Lie algebra $L_{res}(X)$ endowed with the $(X,\tau)$-grading.
\end{prop}

\begin{coro}
\label{grfbasiscoro}
 Under the identification $\gr^\tau F\cong L_{res}(X)$, the set $\ol{C}^\tau_n$ is a $\F_p$-basis for $F_{(\tau,n)}/F_{(\tau,n+1)}$ for all $n\in\N$.
\end{coro}

We return to the study of one-relator pro-$p$-groups. In \cite{JLOneRel}, J.\ Labute showed that if the generating relation of a one-relator pro-$p$-group is 'not too close' to being a $p$-th power, then the group is mild. One key ingredient in his proof is the following theorem, cf. \cite{JLOneRel}, Th.1.

\begin{theo}
\label{Labuteonerelatorlie}
Let $k$ be an arbitrary field and $L(X)$ be the free Lie algebra over $k$ on the set $X=\{X_1,\ldots, X_d\}$ endowed with the $(X,\tau)$-filtration for some $\tau=(\tau_1,\ldots, \tau_d)$. Let $\rho\in L(X)$ be a homogeneous element of degree $n$. Let $\mf{g}=L(X)/(\rho)$ be the quotient of $L(X)$ by the Lie ideal generated by $\rho$. Then the Poincaré series of the enveloping algebra $U_\mf{g}$ of $\mf{g}$ satisfies
\[
U_\mf{g}(t) = \frac{1}{1-(t^{\tau_1}+\ldots + t^{\tau_d}) +t^n}.
\]
\end{theo}

\begin{coro}
\label{onerelatorliecoro}
 Let $X=\{X_1,\ldots, X_d\}$ and $\rho\in L(X)\subset \polyx$ be a non-zero homogeneous element with respect to the $(X,\tau)$-filtration for some $\tau=(\tau_1,\ldots, \tau_d)$ which lies in the free Lie subalgebra $L(X)\subset\polyx$. Then the sequence $\{\rho\}$ is strongly free.
\end{coro}
\begin{proof}
 Let $\mc{R}$ denote the two-sided ideal of $\polyx$ generated by $\rho$. Since $\polyx$ is the universal enveloping algebra of the $\F_p$-Lie algebra $L(X)$, by \cite{NB}, Ch.I, §2.3, Prop.3 there is a canonical isomorphism
\[
 \polyx/\mc{R}\cong U_\mf{g}
\]
where as above $U_\mf{g}$ denotes the enveloping algebra of $L(X)/(\rho)$. Hence the claim follows immediately from \ref{Labuteonerelatorlie}.
\end{proof}

\begin{theo} Let $F$ be the free pro-$p$-group on $x_1,\ldots, x_d$.
\label{onerelatorgrouptheo}
\begin{itemize}
 \item[\rm (i)] Let $1\neq r\in F_{(2)}$ such that for some $\tau=(\tau_1,\ldots, \tau_d)$ the initial form $\rho$ of $r$ in $\gr^\tau F$ is a Lie polynomial, i.e.\ $\rho \in L(X)\subset L_{res}(X)\cong \gr^\tau F,$ \linebreak $X=\{X_1,\ldots, X_d\}$. Then $G=F/(r)$ is mild with respect to the Zassenhaus $(x,\tau)$-filtration. 
 \item[\rm (ii)] Let $r\in F_{(2)}$ such that for some $\tau=(\tau_1,\ldots, \tau_d)$ and some $n\in\N$ with $(n,p)=1$ it holds that $r\in F_{(\tau,n)}\setminus F_{(\tau,n+1)}$. Then $G=F/(r)$ is mild with respect to the Zassenhaus $(x,\tau)$-filtration. 
\end{itemize}
\begin{proof}
 The first statement is a direct consequence of \ref{onerelatorliecoro}. Hence it remains to show (ii): Assume that $r\in F_{(\tau,n)}\setminus F_{(\tau,n+1)}$. Then for the initial form $\rho$ of $r$ in $\gr^\tau F$ we have $\rho\in \F_{(\tau, n)}/F_{(\tau,n+1)}\cong L_{res}(X)^\tau_n$. By \ref{hallbasiscoro}, a $\F_p$-basis of $L_{res}(X)^\tau_n$ is given by the set $\ol{C}^\tau_n\subset \ol{\mc{C}}$. However, since $n$ is prime to $p$, it follows that $\ol{C}^\tau_n\subset \mc{C}$ consists of Hall commutators only, in particular $\rho$ is a Lie polynomial. Now the claim follows from (i).
\end{proof}
\end{theo}

As a special case of (ii), we obtain the following

\begin{coro}
\label{onerelatorzassenhaus}
 Let $G$ be a one-relator pro-$p$-group such that the Zassenhaus invariant $\z(G)$ is prime to $p$. Then $G$ is mild with respect to the Zassenhaus filtration.
\end{coro}

The above corollary can be considered as a more general version of \cite{JLOneRel}, Th.4. In our proof it was essential to combine Labute's result \ref{Labuteonerelatorlie} on quotients of free Lie algebras by a single relation with the fact that we have a complete description of the (graded) structure of $\gr F$ in terms of a free restricted Lie algebra. As an immediate consequence, we obtain the following: If $G$ is a one-relator pro-$p$-group with $\cd G>2$, then $p\mid \z(G)$. This fact is related to an open question which has been originally asked by Serre in \cite{JSQuestion}. The following slightly corrected version is due to Gildenhuys, cf.\ \cite{DG}:
\vspace{10pt}

\it ``Let $G$ be a one-relator pro-$p$-group satisfying $\cd G > 2$. Does $G$ admit a presentation of the form $G=F/(u^p)$, i.e.\ is $G$ the quotient of a free pro-$p$-group by a $p$-th power?''\rm
\vspace{10pt}

A positive answer would imply that for one-relator pro-$p$-groups $G$ we have $\cd G\in \{2,\infty\}$. This question is a pro-$p$-version of an analogous result proven by Lyndon for discrete groups. Unfortunately, we cannot give an answer, however we can provide some evidence by showing that for a one-relator pro-$p$-group with $\z(G)=p$ and $\cd G > 2$, the generating relation is always congruent to a $p$-th power modulo elements of Zassenhaus degree $>p$. 

To this end we need the following 

\begin{lemm}
 \label{linearmap}
Assume that $G$ is a finitely generated pro-$p$-group with $\z(G)\ge n$. Then the $n$-fold Massey product
\[
 \langle \cdot,\ldots, \cdot \rangle_n: H^1(G)^n\longrightarrow H^2(G)
\]
induces a linear map
\[
 B_n: H^1(G)\longrightarrow H^2(G),\ \chi\longmapsto \langle \chi,\chi,\ldots, \chi\rangle_n.
\]
If $n$ is not a power of $p$, then $B_n$ is identically zero. 
\end{lemm}
\begin{proof}
 Let $\chi, \psi\in H^1(G)$. Then by multilinearity of the Massey product, we have
\begin{eqnarray*}
 B_n(\chi+\psi) & = & \langle \chi_n,\ldots, \chi_n\rangle_n\\
& + & \langle \psi, \chi, \ldots, \chi \rangle_n + \langle \chi, \psi, \chi, \ldots, \chi \rangle_n \ldots + \langle \chi,\ldots, \chi, \psi \rangle_n\\
& + & \langle \psi, \psi, \chi, \ldots, \chi \rangle_n + \langle \psi, \chi, \psi, \chi, \ldots, \rangle_n + \ldots + \langle \chi,\ldots, \chi, \psi, \psi \rangle_n\\
& + & \ldots\\
& + & \langle \chi, \psi, \ldots, \psi \rangle_n + \langle \psi, \chi, \psi, \ldots, \psi \rangle_n \ldots + \langle \psi, \ldots, \psi,\chi \rangle_n\\
& + & \langle \psi, \ldots, \psi \rangle_n.
\end{eqnarray*}
By the shuffle identities \ref{shuffleproperties} it follows that in the above sum all lines except for the first and last ones are zero, hence $B_n(\chi+\psi) = B_n(\chi) + B_n(\psi)$. Now write $n=mp^k$ where $(m,p)=1$ and assume $m>1$. Applying the $(p^k,(m-1)p^k)$-shuffle identity with respect to the entries $\xi_{1}=\ldots=\xi_n=\chi$ we find that 
\[
 \tbinom{mp^k}{p^k} \langle \chi,\ldots, \chi \rangle_n = 0.
\]
Since $p\nmid \binom{mp^k}{p^k}$, this implies the $\langle \chi,\ldots, \chi \rangle_n = 0$.
\end{proof}

\begin{rema}
For $n=p$, it can be shown that $B_p$ equals $-B$, where $B$ denotes the \f{Bockstein homomorphism} $B: H^1(G)\longrightarrow H^2(G)$, i.e.\ the connecting homomorphism associated to the short exact sequence
\[
\begin{tikzpicture}[description/.style={fill=white,inner sep=2pt}, bij/.style={below,sloped,inner sep=1.5pt}]
\matrix (m) [matrix of math nodes, row sep=1.5em,
column sep=2.5em, text height=1.5ex, text depth=0.25ex]
{  0 & \Z/p\Z & \Z/p^2\Z & \Z/p\Z & 0,\\};
\path[->,font=\scriptsize]
(m-1-1) edge node[auto] {} (m-1-2)
(m-1-2) edge node[auto] {$p$} (m-1-3)
(m-1-3) edge node[auto] {} (m-1-4)
(m-1-4) edge node[auto] {} (m-1-5);
\end{tikzpicture}
\]
(see \cite{NSW}, Prop.3.9.14 for $p=2$ and \cite{DV2}, Prop.1.2.15 for the general case).
\end{rema}
 
\begin{prop}
\label{onerelatorapprox}
 Let $G$ be a one-relator pro-$p$-group with generator rank $d=h^1(G)$ and Zassenhaus invariant $\z(G)=p$ and $G=F/(r)$ be a minimal presentation of $G$. Assume that $\cd G > 2$. Then there exists a free basis $x_1,\ldots, x_d$ of $F$ and $y\in F$ such that
\[
 r \equiv x_1^p\modulo F_{(p+1)}.
\]
\end{prop}
\begin{proof}
By definition of the Zassenhaus invariant $\z(G)$, we have $r\in F_{(p)}\setminus F_{(p+1)}$. Let $\rho\in F_{(p)}/F_{(p+1)}$ denote the initial form of $r$. First choose an arbitrary basis $\tilde{x}=\{\tilde{x}_1,\ldots, \tilde{x}_d\}$ of $F$ and denote by $\tilde{X}=\{\tilde{X}_1,\ldots, \tilde{X}_d\}\in \gr F$ the corresponding initial forms. Suppose that the homomorphism $B_p: H^1(G)\longrightarrow H^2(G)$ given in \ref{linearmap} is zero, then by \ref{epsilonmaps},
\[
 \varepsilon_{I,p}(r) = 0
\]
for any multi-index of length $p$ of the form $I=(i,i,\ldots, i),\ i=1,\ldots d$ (where the map $\varepsilon_{I,p}$ is defined with respect to the basis $\tilde{x}$) and hence $\rho$ contains no summand of the form $\tilde{X}_i^p,\ i=1,\ldots, d$. Since by \ref{grfbasiscoro}, a basis of $F_{(p)}/F_{(p+1)}$ is given by 
\[
\{\tilde{X}_i^p\ |\ i=1,\ldots, d\}\cup C_p,
\]
where $C_p$ are the Hall commutators of weight $p$, this implies that $\rho$ is a Lie polynomial and consequently using \ref{onerelatorgrouptheo}(i), we have $\cd G = 2$ contradicting the assumption. Therefore, $B_p$ is not the zero map (and hence surjective). We choose a basis $\chi_1,\ldots, \chi_d$ of $H^1(G)$ such that $\chi_2,\ldots, \chi_d$ is a basis of $\ker B_p$. Let $x=\{x_1,\ldots, x_d\}$ be a basis of $F$ lifting the corresponding dual basis of $H^1(G)^\vee = H^1(F)^\vee = F/F_{(2)}$ and denote by $X_1,\ldots, X_d\in \gr F$ the corresponding initial forms. Since $B_p(\chi_1)\neq 0$, it follows that $\tr_r(B_p(\chi_1))\neq 0$ and after replacing $x_1$ by $x_1^a$ for some $a\in \F_p^\times$ we may assume that 
\[
 \varepsilon_{I,p}(r) = 1,\ I=(1,1,\ldots, 1)
\]
and $\varepsilon_{I,p}(r) = 0$ for the multi-indices $I=(i,i,\ldots, i),\ i=2,\ldots, d$ (where $\varepsilon_{I,p}$ is now defined with respect to $x$). Hence,
\[
 \rho = X_1^p + C
\]
where $C$ is a (possibly empty) sum of Hall commutators on $X$ of weight $p$. In other words, 
\[
 r = x_1^p\ c\  r'
\]
where $c$ is a (possibly empty) product of Hall commutators of weight $p$ (in the group $F$) and $r'\in F_{(p+1)}$. We claim that $c=1$. To this end, we consider a different filtration on $F$, namely the Zassenhaus $(x,\tau)$-filtration of $F$ where
 \[
 \tau_i=\left\{ \begin{array}{ll} 
a,& \mbox{if}\ i=1 ,\\
b,& \mbox{else}\\
\end{array}\right.
\]
for natural numbers $a,b$ satisfying $a>b,\ a<\frac{bp}{p-1}$. For $f\in F$, we denote by $\omega_\tau(f)$ the degree of $f$ with respect to this filtration, i.e.\ $f\in F_{(\tau, \omega_\tau (f))}\setminus F_{(\tau, \omega_\tau (f)+1)}$. Suppose $c\neq 1$. Since every Hall commutator in $c$ contains the generator $x_1$ at most $p-1$ times, we have
\[
 \omega_{\tau} (c) \le a(p-1) + b.
\]
Furthermore, $\omega_\tau (x_1^p) = pa > a(p-1)+b$ and $\omega_{\tau} (r') \ge (p+1)\min(a,b)=(p+1)b > a(p-1)+b$. Hence the initial form $\tilde{\rho}\in \gr^\tau F$ of $r$ is a (non-zero) sum of Hall commutators, in particular $\tilde{\rho}$ is a Lie polynomial. Hence, again using \ref{onerelatorgrouptheo}(i), $G$ is mild with respect to the Zassenhaus $(x,\tau)$-filtration, which yields a contradiction to the assumption $\cd > 2$. Therefore, $c=1$ and hence $r\equiv x_1^p \mod F_{(p+1)}$.
\end{proof}

A one-relator pro-$p$-group $G$ with Zassenhaus invariant $\z(G)$ divisible by $p$ can still be mild and hence of cohomological dimension $2$ in many cases, even if the initial form of the generating relation is not a Lie polynomial. We will investigate this in more detail in the next section.

\section{Groups of Demu\v{s}kin type}

We introduce the following abbreviation: For $\chi_1, \ldots, \chi_k\in H^1(G)$ and \linebreak $e_1,\ldots, e_k\ge 1$, we denote by $(\chi_1^{e_1}, \chi_2^{e_2},\ldots, \chi_k^{e_k})$ the vector
\[
 (\chi_1^{e_1}, \chi_2^{e_2},\ldots, \chi_k^{e_k})=(\underbrace{\chi_1,\ldots, \chi_1}_{e_1\ \mbox{\tiny entries}},\underbrace{\chi_2,\ldots, \chi_2}_{e_2\ \mbox{\tiny entries}},\ldots, \underbrace{\chi_k,\ldots, \chi_k}_{e_k\ \mbox{\tiny entries}})\in H^1(G)^{e_1+\ldots + e_k}.
\]

\begin{defi}
\label{Demusikindefi}
 Let $G$ be a one-relator pro-$p$-group with $n=\z(G)$. We say that $G$ is of \f{Demu\v{s}kin type}, if for any $\chi\in H^1(G)\setminus \{0\}$ there exist $\psi\in H^1(G)$ and $e\in\N,\ 1\le e\le n$ such that
\[
 \langle (\chi^{e-1}, \psi, \chi^{n-e}) \rangle_n\neq 0.
\]
\end{defi}

 If $\z(G)=2$, the condition given in \ref{Demusikindefi} means that the cup-product pairing $H^1(G)\times H^1(G)\to H^2(G)$ is non-degenerate, i.e.\ $G$ is a Demu\v{s}kin group. In order to show an infinite group of Demu\v{s}kin type $G$ is of cohomological dimension $\cd G = 2$, we first prove the following

\begin{lemm}
\label{shiftinglemma}
 Let $G$ be a finitely generated pro-$p$-group with $n=\z(G)$. Assume that for some $\chi, \psi\in H^1(G)$ and $1\le e\le n$ it holds that
\[
 \langle (\chi^{e-1}, \psi, \chi^{n-e})\rangle_n \neq 0.
\]
Then also
\[
 \langle (\psi, \chi^{n-1}) \rangle_n \neq 0.
\]
\end{lemm}
\begin{proof}
 If $e=1$ we are done. So assume $e > 1$. Applying the $(e-1,n+1-e)$-shuffle identity with entries $\xi_{i}=\chi,\ i\neq e$ and $\xi_{e} = \psi$ we have
\[
 0 = \langle (\chi^{e-1}, \psi, \chi^{n-e})\rangle_n + \sum_{k=1}^{e-1} \tbinom{n+2-k}{e-k} \langle (\chi^{k-1}, \psi, \chi^{n-k}) \rangle_n.
\]
Since $\langle (\chi^{e-1}, \psi, \chi^{n-e})\rangle_n\neq 0$, we find that $\langle (\chi^{k-1}, \psi, \chi^{n-k})\rangle_n\neq 0$ for some $k<e$.  By induction, the claim follows.
\end{proof}

\begin{theo}
\label{Demuskinmain}
Let $G$ be a pro-$p$-group of Demu\v{s}kin type with $n=\z(G)$. If $G$ is infinite, then $G$ is mild (with respect to the Zassenhaus filtration) and in particular $\cd G=2$. If $G$ is finite, then $n=p^k$ is a $p$-th power and $G\cong \Z/p^k\Z$. 
\end{theo}
\begin{proof}
 A one-relator pro-$p$-group $G$ is finite if and only if $h^1(G)=1$. In the latter case, $n=p^k$ for some $k\ge 1$ and $G\cong \Z/p^k\Z$. Now assume that $h^1(G)>1$. By \ref{linearmap}, we have the homomorphism $B_n: H^1(G)\longrightarrow H^2(G)$. By reason of dimension, its kernel is non-trivial, i.e.\ $\langle (\chi^n) \rangle_n=0$ for some $0\neq \chi\in H^1(G)$. Since $G$ is of Demu\v{s}kin type, there exists $\psi\in H^1(G)$ such that $\chi,\psi$ are $\F_p$-linearly independent and 
 \[
 \langle (\chi^{e-1}, \psi, \chi^{n-e}) \rangle_n\neq 0
\]
 for some $1\le e\le n$. By \ref{shiftinglemma}, we conclude that $\langle (\psi,\chi^{n-1}) \rangle_n\neq 0$. Let $V\subset H^1(G)$ be the $\F_p$-subspace spanned by $\chi$ and $U\subset H^1(G)$ be a subspace containing $\psi$ such that $H^1(G)=U\oplus V$. Then, the $n$-fold Massey product $\langle \cdot, \ldots, \cdot \rangle_n$ satisfies the conditions of \ref{cohomologicalcrit} for $e=1$ and the claim follows.
\end{proof}

\begin{rema}\quad
\begin{itemize}
 \item[\rm (i)] If $n=\z(G)$ is prime to $p$, the above statement is already contained in \ref{onerelatorgrouptheo} and the condition of $G$ being of Demu\v{s}kin type is dispensable.
 \item[\rm (ii)] For $\z(G)=2$, the above theorem yields a new proof that infinite Demu\v{s}kin groups are of cohomological dimension $2$. Again, this is significant only for $p=2$ since for odd $p$ it is sufficient to assume that the cup-product $H^1(G)\times H^1(G)\longrightarrow H^2(G)$ is non-zero in order to show that $G$ is mild.
\end{itemize}
\end{rema}

\begin{exam}
\label{Demuskintypeexa}
Let $G = F/(r)$ where $F$ is the free pro-$3$-group on $x_1,x_2,x_3$ and
\[
 r\equiv x_1^3\ x_2^3\ [[x_1,x_3],x_3] \mod F_{(4)}.
\]
Then $G$ is of Demu\v{s}kin type. By \ref{Demuskinmain} it follows that $\cd G=2$. 
\end{exam}

\bibliographystyle{plain}
\bibliography{promotion}

\vspace{20pt}
\address{\noindent Mathematisches Institut\\ Universität Heidelberg\\ Im Neuenheimer Feld 288\\ 69120 Heidelberg\\ Germany\\ \phantom{1} \\e-mail: gaertner@mathi.uni-heidelberg.de}

\end{document}